\documentclass[reqno]{amsart}
\usepackage{amsmath}
\usepackage{amssymb}
\usepackage{amsthm}
\usepackage{cases}
\usepackage{mathtools}
\mathtoolsset{showonlyrefs=true}
\usepackage{mathrsfs}
\usepackage{hyperref}

\title{An example of $A_2$ Rogers-Ramanujan bipartition identities of level 3}

\author{Shunsuke Tsuchioka}
\address{Department of Mathematical and Computing Sciences, Institute of Science Tokyo, Tokyo 152-8551, Japan}
\email{tshun@kurims.kyoto-u.ac.jp}

\date{Dec 4, 2024}
\keywords{integer partitions,
Rogers-Ramanujan identities,
vertex operators,
cylindric partitions,
W algebras}
\subjclass[2020]{Primary~11P84, Secondary~17B69}

\usepackage{tikz}

\usetikzlibrary{arrows}
\tikzstyle{every picture}+=[remember picture]
\tikzstyle{na} = [baseline=-.5ex]
\tikzstyle{mine}= [arrows={angle 90}-{angle 90},thick]

\def\Llleftarrow{%
\lower2pt\hbox{\begingroup
\tikz
\draw[shorten >=0pt,shorten <=0pt] (0,3pt) -- ++(-1em,0) (0,1pt) -- ++(-1em-1pt,0) (0,-1pt) -- ++(-1em-1pt,0) (0,-3pt) -- ++(-1em,0) (-1em+1pt,5pt) to[out=-105,in=45] (-1em-2pt,0) to[out=-45,in=105] (-1em+1pt,-5pt);
\endgroup}
}

\newtheorem{Thm}{Theorem}[section]
\newtheorem{Def}[Thm]{Definition}

\newtheorem{Prop}[Thm]{Proposition}
\newtheorem{Lem}[Thm]{Lemma}

\newtheorem{Rem}[Thm]{Remark}
\newtheorem{Cor}[Thm]{Corollary}
\newtheorem{Ex}[Thm]{Example}
\newcommand{\YY}{Y}
\newcommand{\CP}[1]{\mathcal{C}_{#1}}
\newcommand{\LC}[1]{\mathcal{C}^{(#1)}}
\newcommand{\YP}[1]{Y_{+}(#1)}
\newcommand{\YM}[1]{Y_{-}(#1)}

\newcommand{\YMP}[1]{Y_{\pm}(#1)}

\newcommand{\GEE}{\mathfrak{g}}
\newcommand{\GE}{\widetilde{\mathfrak{g}}}
\newcommand{\GAA}{\widetilde{\mathfrak{a}}}
\newcommand{\GAAA}{\widehat{\mathfrak{a}}}

\newcommand{\LEXO}[1]{<^{#1}_{\mathsf{lex}}}
\newcommand{\LEXXO}[1]{\leq^{#1}_{\mathsf{lex}}}
\DeclareMathOperator{\LEX}{<_{\mathsf{lex}}}

\DeclareMathOperator{\LEXX}{\leq_{\mathsf{lex}}}
\DeclareMathOperator{\SHAPE}{\mathsf{shape}}
\DeclareMathOperator{\SEQ}{\mathsf{Seq}}
\DeclareMathOperator{\RED}{\mathsf{Irr}}
\DeclareMathOperator{\ZSEQ}{\mathsf{FS}}
\DeclareMathOperator{\NSEQ}{\mathsf{FS}_0}
\DeclareMathOperator{\SORT}{\mathsf{sort}}
\DeclareMathOperator{\MAT}{Mat}
\DeclareMathOperator{\SPAN}{\mathsf{span}}

\DeclareMathOperator{\TPAR}{\mathsf{Par}_{2\mathsf{color}}}
\DeclareMathOperator{\IND}{\mathsf{Ind}}
\DeclareMathOperator{\MAP}{\mathsf{Map}}

\DeclareMathOperator{\SHIFT}{\mathsf{shift}}
\DeclareMathOperator{\BIR}{\mathsf{R}}
\DeclareMathOperator{\BIRP}{\mathsf{R}'}

\DeclareMathOperator{\TR}{\mathsf{tr}}

\DeclareMathOperator{\DEG}{\mathsf{deg}}
\DeclareMathOperator{\END}{\mathsf{End}}

\newcommand{\ggeq}{\mathrel{\underline{\gg}}}
\newcommand{\lleq}{\mathrel{\underline{\ll}}}
\DeclareMathOperator{\NAANO}{\vartriangleleft}
\DeclareMathOperator{\AANO}{\trianglelefteq}
\DeclareMathOperator{\ANO}{\gg}
\DeclareMathOperator{\ANOE}{\ggeq}

\DeclareMathOperator{\ANOO}{\lleq}
\DeclareMathOperator{\SUPP}{\mathsf{Supp}}
\DeclareMathOperator{\COLOR}{\mathsf{color}}
\DeclareMathOperator{\CONT}{\mathsf{cont}}
\DeclareMathOperator{\CONTT}{\mathsf{cont}}
\DeclareMathOperator{\LENGTH}{\ell}
\DeclareMathOperator{\LM}{\mathsf{LM}}
\DeclareMathOperator{\PAI}{\pi^{\bullet}}
\newcommand{\HPRIN}{H^{\textrm{prin}}}
\newcommand{\HPPRIN}{H^{\rm{prin}}}
\DeclareMathOperator{\isom}{\to}
\newcommand{\SUM}[1]{|#1|}
\newcommand{\cint}[1]{{\underline{#1}}_c}
\newcommand{\EXTE}[1]{\widehat{#1}}
\newcommand{\REG}[1]{\mathsf{CReg}_{#1}}
\newcommand{\JO}[1]{$q_{#1}$}
\newcommand{\JOO}[1]{q_{#1}}
\newcommand{\ENU}{N}
\newcommand{\KEI}[1]{K_{#1}}
\newcommand{\AOX}[1]{X(#1)}
\newcommand{\AOB}[1]{B(#1)}
\newcommand{\emptypar}{\emptyset}
\newcommand{\mya}{a}
\newcommand{\myb}{b}
\newcommand{\myc}{c}
\newcommand{\myd}{d}
\newcommand{\mye}{e}
\newcommand{\myf}{f}
\newcommand{\myg}{g}
\newcommand{\myh}{h}
\newcommand{\myi}{i}
\newcommand{\myj}{j}
\newcommand{\myk}{k}
\newcommand{\myl}{l}
\newcommand{\mym}{m}
\newcommand{\myphi}{\Psi}
\newcommand{\TTE}{\widetilde{\varepsilon}}
\newcommand{\TTP}{\widetilde{p}}
\newcommand{\TTQ}{\widetilde{q}}
\newcommand{\TTR}{\widetilde{r}}
\newcommand{\TTS}{\widetilde{s}}
\newcommand{\EMPTYWORD}{\varepsilon}
\newcommand{\NONC}{\mathcal{E}}
\newcommand{\CZ}{\cint{\mathbb{Z}}}
\newcommand{\AZ}{\mathbb{Z}_{2\mathsf{color}}}
\newcommand{\MP}{P^{+}}
\newcommand{\VZ}{v}
\newcommand{\VU}{u}
\newcommand{\BUI}[1]{Y(#1)v}
\newcommand{\BUII}[2]{Y(#1)v_{#2}}
\newcommand{\BUIII}[1]{Y(#1)w}
\newcommand{\RR}{\mathsf{RR}_i}

\usepackage{graphicx}

\begin{document}
\maketitle

\begin{abstract}
We give manifestly positive Andrews-Gordon type series for the
level 3 standard modules of the affine Lie algebra of type $A^{(1)}_2$.
We also give corresponding bipartition identities,
which have representation theoretic interpretations via the vertex operators.
Our proof is based on
the Borodin product formula, the Corteel-Welsh recursion for the cylindric partitions,
a $q$-version of Sister Celine's technique
and a generalization of Andrews' partition ideals by finite automata due to Takigiku and the author. 
\end{abstract}

\section{Introduction}
\subsection{The Rogers-Ramanujan partition identities}
A partition $\lambda=(\lambda_1,\dots,\lambda_{\ell})$ of a nonnegative integer $n$ is 
a weakly decreasing sequence of positive integers (called parts) whose sum $|\lambda|=\lambda_1+\dots+\lambda_{\ell}$ 
(called the size) is $n$.
The celebrated Rogers-Ramanujan partition identities are stated as follows. 
\begin{quotation}\label{eq:RR:PT}
Let $i=1,2$. The number of partitions of $n$ such that parts are at least $i$ and consecutive 
parts differ by at least $2$ is equal to the number of partitions of $n$ 
into parts congruent to $\pm i$ \mbox{modulo $5$.}
\end{quotation}

As $q$-series identities,
the Rogers-Ramanujan identities are stated as 
\begin{align}
    \sum_{n\ge 0} \frac{q^{n^2}}{(q;q)_n}
    = \frac{1}{(q,q^4;q^5)_\infty},
    \quad
    \sum_{n\ge 0} \frac{q^{n^2+n}}{(q;q)_n}
    = \frac{1}{(q^2,q^3;q^5)_\infty},
\label{eq:RR:q}
\end{align}
where, as usual, for $n\in\mathbb{Z}_{\geq0} \sqcup \{\infty\}$, we define
the $q$-Pochhammer symbols by
\begin{align*}
    (a;q)_n = \prod_{0\leq j<n} (1-aq^j),
    \quad
    (a_1,\dots,a_k;q)_{n}
    = (a_1;q)_n \cdots (a_k;q)_n.
\end{align*}

\subsection{The main result}\label{mainse}

A bipartition of a nonnegative integer $n$ is a pair of partitions $\boldsymbol{\lambda}=(\lambda^{(1)},\lambda^{(2)})$ such that $|\lambda^{(1)}|+|\lambda^{(2)}|=n$.
A 2-colored partition $\lambda=(\lambda_1,\dots,\lambda_{\ell})$ is a weakly decreasing sequence of positive integers and colored positive integers (called parts)
with respect to the order
\begin{align}
\cdots>3>\cint{3}>2>\cint{2}>1>\cint{1}.
\label{orde}
\end{align}
We define the size $|\lambda|$ of $\lambda$ by $|\lambda|=\CONT(\lambda_1)+\dots+\CONT(\lambda_{\ell})$,
where $\CONT(k)=\CONT(\cint{k})=k$ 
for a positive integer $k$.
It is evident that there exists a natural identification between the bipartitions of $n$ and the 2-colored partitions of $n$.

We put $\COLOR(k)=+$ and $\COLOR(\cint{k})=-$ for a positive integer $k$.
For a 2-colored partition $\lambda=(\lambda_1,\dots,\lambda_{\ell})$, we consider the following conditions.
\begin{enumerate}
\item[(D1)] If consecutive parts $\lambda_a$ and $\lambda_{a+1}$, where $1\leq a<\ell$, 
satisfy $\CONT(\lambda_a)-\CONT(\lambda_{a+1})\leq 1$, then $\CONT(\lambda_a)+\CONT(\lambda_{a+1})\in 3\mathbb{Z}$ and $\COLOR(\lambda_a)\ne\COLOR(\lambda_{a+1})$.
\item[(D2)] If consecutive parts $\lambda_a$ and $\lambda_{a+1}$, where $1\leq a<\ell$, 
satisfy $\CONT(\lambda_a)-\CONT(\lambda_{a+1})=2$ and $\CONT(\lambda_a)+\CONT(\lambda_{a+1})\not\in 3\mathbb{Z}$, then $(\COLOR(\lambda_a),\COLOR(\lambda_{a+1}))\ne(-,+)$. 
\item[(D3)] $\lambda$ does not contain $(3k,\cint{3k},\cint{3k-2})$, $(3k+2,3k,\cint{3k})$ and $(\cint{3k+2},3k+1,3k-1,\cint{3k-2})$ for $k\geq 1$.
\item[(D4)] $\lambda$ does not contain $1$, $\cint{1}$, and $\cint{2}$ as parts 
(i.e., $\lambda_a\ne 1,\cint{1},\cint{2}$
for $1\leq a\leq\ell$).
\end{enumerate}

On (D3), we say that $\lambda=(\lambda_1,\dots,\lambda_{\ell})$ contains
$\mu=(\mu_1,\dots,\mu_{\ell'})$ if there exists $0\leq i\leq \ell-\ell'$ such that
$\lambda_{j+i}=\mu_j$ for $1\leq j\leq \ell'$.

\begin{Thm}\label{RRbiiden}
Let $\BIR$ (resp. $\BIRP$) the set of 2-colored partitions that satisfy the conditions (D1)--(D3) 
(resp. (D1)--(D4)) above. Then, we have
\begin{align*}
\sum_{\lambda\in\BIR}q^{|\lambda|} &= \frac{(q^2,q^4;q^6)_{\infty}}{(q,q,q^3,q^3,q^5,q^5;q^6)_{\infty}},\\
\sum_{\lambda\in\BIRP}q^{|\lambda|} &= \frac{1}{(q^2,q^3,q^3,q^4;q^6)_{\infty}}
\end{align*}
\end{Thm}

\begin{Ex}\label{exbipar}
We have
\begin{align*}
\BIR &= \{\emptypar,
(1),(\cint{1}),(2),(\cint{2}),(3),(\cint{3}),(2,\cint{1}),(\cint{2},1),
(4),(\cint{4}),(3,1),(3,\cint{1}),(\cint{3},\cint{1}),\dots\},\\
\BIRP &= \{\emptypar,
(2),(3),(\cint{3}),(4),(\cint{4}),(5),(\cint{5}),(6),(\cint{6}),(4,2),(\cint{4},2),(3,\cint{3}),(7),(\cint{7}),(5,2),(\cint{5},2),\dots\}.
\end{align*}
\end{Ex}

We propose the aforementioned bipartition identities (in terms of 2-colored partitions)
as $A_2$ Rogers-Ramanujan bipartition identities of level 3. In the rest of this section,
we give some justifications for our proposal.

\subsection{Lie theoretic interpretations}
Throughout this paper, we denote by $\GEE(A)$ the Kac-Moody Lie algebra associated with
a generalized Cartan matrix (GCM, for short) $A$.
For an affine GCM $A$ and
a dominant integral weight $\Lambda\in\MP$, we denote by $\chi_{A}(V(\Lambda))$
the principal character of the standard module (a.k.a., the integrable highest weight module) $V(\Lambda)$ with a highest 
weight vector $v_{\Lambda}$ of the affine Lie algebra $\GEE(A)$.

In ~\cite{LM}, Lepowsky and Milne observed a similarity between the characters 
of the level 3 standard modules of the affine Lie algebra of type $A^{(1)}_{1}$
\begin{align}
\chi_{A^{(1)}_1}(V(2\Lambda_0+\Lambda_1)) = 
\frac{1}{(q,q^4;q^5)_{\infty}},\quad
\chi_{A^{(1)}_1}(V(3\Lambda_0)) = 
\frac{1}{(q^2,q^3;q^5)_{\infty}}.
\label{a11level3}
\end{align}
and the infinite products of the Rogers-Ramanujan identities \eqref{eq:RR:q}.

This was one of the motivations for inventing the vertex operators~\cite{LW1} as well as ~\cite{FK,Seg}. 
Subsequently, in ~\cite{LW3}, Lepowsky and Wilson promoted the observation \eqref{a11level3} to 
a vertex operator theoretic proof of \eqref{eq:RR:q} (see also ~\cite{LW2}),
which provides a Lie theoretic interpretations of the infinite sums in \eqref{eq:RR:q}.
The result is generalized to higher levels in ~\cite{LW4}, assuimg the Andrews-Gordon-Bressoud identities 
(a generalization of the Rogers-Ramanujan identities, see ~\cite[\S3.2]{Sil}), 
for which Meurman and Primc gave a vertex operator theoretic proof in ~\cite{MP}.

Recall the principal realization of the affine Lie algebra $\GEE(A^{(1)}_{1})$ (see ~\cite[\S7,\S8]{Kac}).
Using the notation in ~\cite[\S2]{MP}, it affords a basis 
\begin{align*}
\{\AOB{n},\AOX{n'},c,d\mid n\in\mathbb{Z}\setminus 2\mathbb{Z}, n'\in\mathbb{Z}\},
\end{align*}
of $\GEE(A^{(1)}_{1})$. Note that
$\{\AOB{n},c\mid n\in\mathbb{Z}\setminus 2\mathbb{Z}\}$ forms 
a basis of the principal Heisenberg subalgebra of $\GEE(A^{(1)}_{1})$.
The following is essentially the Lepowsky-Wilson interpretation of 
the Rogers-Ramanujan partition identities
in terms of the representation theory of $\GEE(A^{(1)}_1)$ (see also \cite[Theorem 10.4]{LW3} and \cite[Appendix]{MP}).

\begin{Thm}[{\cite{LW3,MP}}]
For $i=1,2$, let $\RR$ be the set of partitions such that parts are at least $i$ and consecutive 
parts differ by at least $2$. Then, the set
\begin{align*}
\{\AOB{-\mu_1}\cdots\AOB{-\mu_{\ell'}}\AOX{-\lambda_1}\cdots\AOX{-\lambda_\ell}v_{(i+1)\Lambda_0+(2-i)\Lambda_1}\}
\end{align*}
forms a basis of $V((i+1)\Lambda_0+(2-i)\Lambda_1)$,
where $(\mu_1,\dots,\mu_{\ell'})$ varies in $\REG{2}$
and $(\lambda_1,\dots,\lambda_{\ell})$ varies in $\RR$.
\end{Thm}

Here, we denote by $\REG{p}$ the set of $p$-class regular partitions for $p\geq 2$.
Recall that a partition 
is called $p$-class regular if no parts are divisible by $p$. 

We show a similar interpretation for $\BIR$ and $\BIRP$.
Using the notation in \S\ref{vertset}, 
\begin{align*}
\{
B(n),x_{\alpha_1}(n'),x_{-\alpha_1}(n'),c,d
\mid n\in\mathbb{Z}\setminus3\mathbb{Z},n'\in\mathbb{Z}
\}
\end{align*}
forms a basis of $\GEE(A^{(1)}_2)$ (see also Remark \ref{sl3basis}). 

\begin{Thm}\label{biideninter}
For $i=1$ (resp. $i=2$),
the set
\begin{align*}
\{\AOB{-\mu_1}\cdots\AOB{-\mu_{\ell'}}
x_{\COLOR(\lambda_1)\alpha_1}(-\!\CONTT(\lambda_1))\cdots x_{\COLOR(\lambda_{\ell})\alpha_1}(-\!\CONTT(\lambda_{\ell}))
v_{(2i-1)\Lambda_0+(2-i)\Lambda_1+(2-i)\Lambda_2}\}
\end{align*}
forms a basis of $V((2i-1)\Lambda_0+(2-i)\Lambda_1+(2-i)\Lambda_2)$, where $(\mu_1,\dots,\mu_{\ell'})$ varies in $\REG{3}$
and $(\lambda_1,\dots,\lambda_{\ell})$ varies in $\BIR$ (resp. $\BIRP$).
\end{Thm}

Note that Theorem \ref{biideninter} implies Theorem \ref{RRbiiden} thanks to
\begin{align}
\chi_{A^{(1)}_{2}}(V(\Lambda_0+\Lambda_1+\Lambda_2)) &= 
\frac{(q^2,q^4;q^6)_{\infty}}{(q,q,q^3,q^3,q^5,q^5;q^6)_{\infty}},\\
\chi_{A^{(1)}_{2}}(V(3\Lambda_0)) &= 
\frac{1}{(q^2,q^3,q^3,q^4;q^6)_{\infty}}.
\label{charcalc}
\end{align}

In \S\ref{maincomp} and \S\ref{cal}, we show that
the set in Theorem \ref{biideninter} spans $V((2i-1)\Lambda_0+(2-i)\Lambda_1+(2-i)\Lambda_2)$ (see Corollary \ref{biidenintercor}).
Thus, Theorem \ref{biideninter} and Theorem \ref{RRbiiden} are equivalent.

\subsection{$A_2$ Rogers-Ramanujan identities}\label{agthm}
A standard $q$-series technique to prove the Andrews-Gordon-Bressoud identities is the Bailey Lemma 
(see ~\cite[\S3]{An2} and ~\cite[\S3]{Sil}).
In ~\cite{ASW}, Andrews-Schilling-Warnaar found an $A_2$ analog of it and
obtained a family of Rogers-Ramanujan type identities for characters of the $W_3$ algebra.
The result can be regarded as an $A^{(1)}_2$ analog of the Andrews-Gordon-Bressoud identities, whose infinite products
are some of the principal characters of the standard modules of $\GEE(A^{(1)}_2)$ (see ~\cite[Theorem 5.1, Theorem 5.3, Theorem 5.4]{ASW}) after a multiplication of $(q;q)_{\infty}$.

In our case of level 3, Andrews-Schilling-Warnaar identities are stated as follows.

\begin{Thm}[{\cite[Theorem 5.4 specialized to $k=2$ and $i=2,1$]{ASW}}]
\begin{align*}
\sum_{s,t\geq 0}\frac{q^{s^2-st+t^2}}{(q;q)^2_{s+t}}{s+t \brack s}_{q^3}
&= \frac{1}{(q;q)_{\infty}}\frac{(q^2,q^4;q^6)_{\infty}}{(q,q,q^3,q^3,q^5,q^5;q^6)_{\infty}},\\
\sum_{s,t\geq 0}\frac{q^{s^2-st+t^2+s+t}}{(q;q)_{s+t+1}(q;q)_{s+t}}{s+t \brack s}_{q^3}
&= \frac{1}{(q;q)_{\infty}}\frac{1}{(q^2,q^3,q^3,q^4;q^6)_{\infty}}.
\end{align*}
\end{Thm}

We show manifestly positive, Andrews-Gordon type series (in the sense of ~\cite{TT2})
for $\BIR$ and $\BIRP$ as follows.
Here, the length $\ell(\lambda)$ of a 2-colored partition $\lambda$ is defined to be the number of parts.
For the size $|\lambda|$ of $\lambda$, see \S\ref{mainse}.

\begin{Thm}\label{RRidentAG}
\begin{align*}
\sum_{\lambda\in\BIR}x^{\ell(\lambda)}q^{|\lambda|}
&=
\sum_{a,b,c,d\geq 0}\frac{q^{a^2+b^2+3c^2+3d^2+2ab+3ac+3ad+3bc+3bd+6cd}x^{a+b+2c+2d}}{(q;q)_a(q;q)_b(q^3;q^3)_c(q^3;q^3)_d},\\
\sum_{\lambda\in\BIRP}x^{\ell(\lambda)}q^{|\lambda|}
&=
\sum_{a,b,c,d\geq 0}\frac{q^{a(a+1)+b(b+2)+3c(c+1)+3d(d+1)+2ab+3ac+3ad+3bc+3bd+6cd}x^{a+b+2c+2d}}{(q;q)_a(q;q)_b(q^3;q^3)_c(q^3;q^3)_d}.
\end{align*}
\end{Thm}

The result is simiar to the fact that, for $i=1,2$, we have
\begin{align*}
\sum_{\lambda\in\RR}x^{\ell(\lambda)}q^{|\lambda|}=\sum_{n\geq 0}\frac{q^{n(n+i-1)}x^n}{(q;q)_n}.
\end{align*}

Recently, Kanade-Russell showed (see ~\cite[(1.8)]{KR3})
\begin{align*}
\sum_{s,t\geq 0}\frac{q^{s^2-st+t^2+t}}{(q;q)_{s+t+1}(q;q)_{s+t}}{s+t \brack s}_{q^3}
= \frac{1}{(q;q)_{\infty}}\frac{1}{(q,q^2;q^3)_{\infty}},
\end{align*}
where $(q,q^2;q^3)_{\infty}^{-1}=\chi_{A^{(1)}_2}(2\Lambda_0+\Lambda_1)$.
Although it can be similarly proven  
\begin{align*}
\sum_{a,b,c,d\geq 0}\frac{q^{a^2+b(b+1)+3c^2+c+3d^2+2d+2ab+3ac+3ad+3bc+3bd+6cd}}{(q;q)_a(q;q)_b(q^3;q^3)_c(q^3;q^3)_d}
= \frac{1}{(q,q^2;q^3)_{\infty}},
\end{align*}
we do not consider the level 3 module $V(2\Lambda_0+\Lambda_1)$ in this paper.
A reason is that $V(2\Lambda_0+\Lambda_1+p\delta)$ for $p\in\mathbb{Z}$
is not a submodule of $V(\Lambda_0)^{\otimes 3}$ as in Remark \ref{notsub}.

After the groundbreaking paper ~\cite{ASW}, 
a vast literature is devoted to the study of $A_2$ Rogers-Ramanujan identities 
(see ~\cite{CDA,CW,FFW,FW,KR3,Unc,War2,War,War3}), especially to the search for 
manifestly positive infinite sums such as ~\cite[Theorem 5.2]{ASW}, ~\cite[Theorem 1.1]{CW}
for level 4 and ~\cite[Theorem 1.6]{CDA} for level 5.

After a success of vertex operator theoretic proofs of the Rogers-Ramanujan identities such as ~\cite{LW3,MP},
it has been expected that, for an affine GCM $X^{(r)}_N$ and a dominant integral weight $\Lambda\in\MP$, 
there should exist a Rogers-Ramanujan type identity whose infinite product is given by $\chi_A(V(\Lambda))$. 
It is natural to expect that the sum side is related to
the $n(X^{(r)}_N)$-colored partitions, where
\begin{align*}
  n(X^{(r)}_N)=\frac{\textrm{the number of roots of type $X_N$}}{\textrm{the $r$-twisted Coxeter number of $X_N$}}
=\textrm{the size of $X^{(r)}_N$}.
\end{align*}
Concerning the fact $n(A^{(1)}_{r})=r$, 
it is expected that the
$A_2$ Rogers-Ramanujan identities are related to the 2-colored partitions (and thus to the bipartitions).
It would be interesting if
the results in this paper are generalized to higher levels. 

\hspace{0mm}

\noindent{\bf Organization of the paper.} 
The paper is organized as follows.
In \S\ref{vertset}, we recall the principal realization of $\widehat{\mathfrak{sl}_3}=\GEE(A^{(1)}_2)$ and the vertex operator realization of 
the basic module following ~\cite{LW3}.
In \S\ref{maincomp} and \S\ref{cal}, we show that the definiting conditions for $\BIR$ and $\BIRP$
are naturally deduced by calculations (similar to ~\cite{Cap0,MP,Nan}) of the vertex operators on the triple tensor product of the basic module.
In \S\ref{auto}, we show that $q$-difference equations for $\BIR$ and $\BIRP$ are automatically
derived by the technique developed in ~\cite{TT} as a generalization of Andrews' 
linked partition ideals~\cite[Chapter 8]{An1} using finite automata.
In \S\ref{ags}, we briefly review
certificate recurrences, which are obtained by
a $q$-version of Sister Celine's technique~\cite{Koe,Rie,WZ} for a $q$-proper hypergeometric term.
It
gives automatically a $q$-difference equation for an Andrews-Gordon
type series and thus a proof of Theorem \ref{RRidentAG}.
These and corresponding results for the cylindric partitions (see \S\ref{gs111})
give a proof of Theorem \ref{RRbiiden} (see \S\ref{finalsec}) by combining the 
standard results, such as
the Corteel-Welsh recursion~\cite{CW} and the Borodin product formula~\cite{Bor},
which are reviewed in \S\ref{cylin}.



\section{The vertex operators}\label{vertset}
In this section, we fix $m=3$ and $\omega=\exp(2\pi\sqrt{-1}/m)$.
As usual, the affine Cartan matrix
\begin{align}
A^{(1)}_2=
\begin{pmatrix}
2  & -1 & -1 \\
-1 &  2 & -1 \\
-1 & -1 & 2
\end{pmatrix}
\label{A12GCM}
\end{align}
is indexed by the set $I=\{0,1,2\}$. Note that $A_2=(A^{(1)}_2)|_{I_0\times I_0}$, where $I_0=\{1,2\}$.

\subsection{The principal realization of the affine Lie algebra $\GEE(A^{(1)}_2)$}\label{prafflie}
We regard 
\begin{align*}
\mathfrak{sl}_3=\{M\in\MAT_3(\mathbb{C})\mid\TR M=0\}
\end{align*}
with the Cartan-Killing form $\langle M_1,M_2\rangle=\TR(M_1M_2)$ 
as the Kac-Moody Lie algebra $\GEE(A_2)$ 
with the Chevalley generators $e_1=E_{12}$, $e_2=E_{23}$, $f_1=E_{21}$, $f_2=E_{32}$, $h_1=E_{11}-E_{22}$ and $h_2=E_{22}-E_{33}$,
where $E_{ij}$ is the $3\times 3$ matrix unit for $1\leq i,j\leq 3$.

Take the principal automorphism
\begin{align*}
\nu:\mathfrak{sl}_3\to\mathfrak{sl}_3,\quad
E_{ij}\mapsto \omega^{j-i}E_{ij}.
\end{align*}
Note that $\langle,\rangle$ is $\nu$-invariant, i.e., $\langle\nu M_1,\nu M_2\rangle=\langle M_1,M_2\rangle$ for $M_1,M_2\in\mathfrak{sl}_3$.
The $\nu$-twisted affinization is given by
\begin{align*}
\GE &= 
(\SPAN\{E_{11}-E_{22},E_{22}-E_{33}\}\otimes\mathbb{C}[t^3,t^{-3}])
\oplus(\SPAN\{E_{12},E_{23},E_{31}\}\otimes t\mathbb{C}[t^3,t^{-3}])\\
&\quad\quad
\oplus(\SPAN\{E_{13},E_{21},E_{32}\}\otimes t^2\mathbb{C}[t^3,t^{-3}])
\oplus\mathbb{C}c\oplus\mathbb{C}d,
\end{align*}
with the Lie algebra structure 
\begin{align*}
[M\otimes t^n,M'\otimes t^{n'}]=[M,M']\otimes t^{n+n'}+\frac{n \langle M,M'\rangle}{3}\delta_{n+n',0}c, \quad
[d,M\otimes t^n]=nM\otimes t^n,
\end{align*}
where $M,M'\in\mathfrak{sl}_3$, $n,n'\in\mathbb{Z}$, and $c$ is central (see ~\cite[(2.17)]{LW3}).

The principal realization is the Lie algebra isomorphism $\GEE(A^{(1)}_2)\isom \GE$ given by
\begin{align*}
\begin{array}{lll}
e_0\mapsto E_{31}\otimes t, &
e_1\mapsto E_{12}\otimes t, &
e_2\mapsto E_{23}\otimes t,\\
f_0\mapsto E_{13}\otimes t^{-1}, &
f_1\mapsto E_{21}\otimes t^{-1}, &
f_2\mapsto E_{32}\otimes t^{-1},\\
h_0\mapsto (E_{33}-E_{11})\otimes 1+\frac{c}{3}, &
h_1\mapsto (E_{11}-E_{22})\otimes 1+\frac{c}{3}, &
h_2\mapsto (E_{22}-E_{33})\otimes 1+\frac{c}{3},
\end{array}
\end{align*}
in addition to $d\mapsto d$.
The principal degeree on $\GEE(A^{(1)}_2)$ (and thus on $\GE$) is defined by 
$\DEG e_i=1=-\DEG f_i$ and $\DEG h_i=0=\DEG d$ for $i\in I$. See also ~\cite[\S7,\S8]{Kac}.

\subsection{The principal Cartan subalgebra}

Let $E_{(3)}=E_{12}+E_{23}+E_{31}$ be the 
principal cyclic element due to Kostant~\cite{Kos}. 
Take the principal Cartan subalgebra
\begin{align*}
\HPRIN = \{M\in \mathfrak{sl}_3\mid [M,E_{(3)}]=O\}=\mathbb{C}E_{(3)}\oplus\mathbb{C}E^2_{(3)}.
\end{align*}

We have a root space decomposition (see ~\cite[Corollary 26.2.7]{DBT})
\begin{align}
  \mathfrak{sl}_3 = \HPRIN \oplus \bigoplus_{1\leq s\ne t\leq 3}\mathbb{C}A_{st},\quad
  \textrm{where}\quad
  A_{st}=\frac{1}{3}\sum_{1\leq i,j\leq 3}\omega^{jt-is}E_{ij}.
\label{rootdec}
\end{align}
By ~\cite[Lemma 26.2.4]{DBT}, for $1\leq s,t\leq 3$, we have 
\begin{align*}
  [h,A_{st}] = \langle h,B_{st} \rangle A_{st},\quad
  \textrm{where}\quad
  B_{st}=\frac{1}{3}\sum_{k=1}^{2}(\omega^{-ks}-\omega^{-kt})E_{(3)}^{3-k}.
\end{align*}

In the rest of the paper, we put
\begin{align}
\begin{split}
  \alpha_1
  &= B_{12}
  = \frac{\omega-\omega^2}{3}E_{(3)}+\frac{\omega^2-\omega}{3}E_{(3)}^2,\\
  \alpha_2 &= B_{23}=\frac{\omega^2-1}{3}E_{(3)}+\frac{\omega-1}{3}E_{(3)}^2,\\ 
  \Phi &= \{\alpha_1,\alpha_2,\alpha_1+\alpha_2,-\alpha_1,-\alpha_2,-(\alpha_1+\alpha_2)\}(\subseteq\HPRIN),\\
  x_{\alpha_1} &= A_{12}=\begin{footnotesize}
  \frac{1}{3}
  \begin{pmatrix}
    \omega & 1 & \omega^2 \\
    1 & \omega^2 & \omega \\
    \omega^2 & \omega & 1
  \end{pmatrix}\end{footnotesize},\quad
  x_{-\alpha_1} = A_{21}=\begin{footnotesize}
  \frac{1}{3}
  \begin{pmatrix}
    \omega^2 & 1 & \omega \\
    1 & \omega & \omega^2 \\
    \omega & \omega^2 & 1
  \end{pmatrix}\end{footnotesize},\\
  x_{\alpha_2} &= A_{23}=\begin{footnotesize}
  \frac{1}{3}
  \begin{pmatrix}
    \omega & \omega & \omega \\
    \omega^2 & \omega^2 & \omega^2 \\
    1 & 1 & 1
  \end{pmatrix}\end{footnotesize},\quad
  x_{-\alpha_2} = A_{32} =\begin{footnotesize}
  \frac{1}{3}
  \begin{pmatrix}
    \omega^2 & \omega & 1 \\
    \omega^2 & \omega & 1 \\
    \omega^2 & \omega & 1
  \end{pmatrix}\end{footnotesize},\\
  x_{\alpha_1+\alpha_2} &= A_{13}=\begin{footnotesize}
  \frac{1}{3}
  \begin{pmatrix}
    \omega^2 & \omega^2 & \omega^2 \\
    \omega & \omega & \omega \\
    1 & 1 & 1
  \end{pmatrix}\end{footnotesize},\quad
  x_{-(\alpha_1+\alpha_2)} = A_{31}  
  =\begin{footnotesize}
  \frac{1}{3}
  \begin{pmatrix}
    \omega & \omega^2 & 1 \\
    \omega & \omega^2 & 1 \\
    \omega & \omega^2 & 1
  \end{pmatrix}\end{footnotesize}.
\end{split}
\label{rootde2}
\end{align}
Note that we have $\nu(\alpha_1)=\alpha_2$, $\nu(\alpha_2)=-(\alpha_1+\alpha_2)$.
Note also that we have $\nu(x_{\beta})=x_{\nu\beta}$ for $\beta\in\Phi$, which is easily seen by ~\cite[(26.2.17),(26.3.4),(26.3.5)]{DBT}.

\subsection{The principal Heisenberg subalgebra}
For $M\in\mathfrak{sl}_3$ and $n\in\mathbb{Z}$, we put
\begin{align}
M(n)=\pi_{(n)}(M)\otimes t^n\in\GE,
\label{emurecall}
\end{align}
where $\pi_{(n)}:\mathfrak{sl}_3\to\{M\in\mathfrak{sl}_3\mid\nu(M)=\omega^nM\}$ is the projection.


Let $\HPRIN_{(n)}=\pi_{(n)}(\HPRIN)$. Then, $\HPRIN_{(n)}=\{0\}$ when $n\in 3\mathbb{Z}$. For $n\in\mathbb{Z}\setminus 3\mathbb{Z}$,  
$B(n)=E_{(3)}^n\otimes t^n$
is a basis of the 1-dimensional space $\HPRIN_{(n)}$, where $B=E_{(3)}+E^2_{(3)}$.
Note that we have $[B(n),B(n')]=nc\delta_{n+n',0}$ for $n,n'\in\mathbb{Z}\setminus 3\mathbb{Z}$.

\begin{Rem}\label{sl3basis}
It is not difficult to see (e.g., by $\pi_{(n)}(\nu x_{\beta})=\omega^n\pi_{(n)}(x_{\beta})$) that 
\begin{align*}
\{
B(n),x_{\alpha_1}(n'),x_{-\alpha_1}(n'),c,d
\mid n\in\mathbb{Z}\setminus3\mathbb{Z},n'\in\mathbb{Z}
\}
\end{align*}
forms a basis of $\GE$ with $\DEG B(n)=n$ and $\DEG x_{\pm\alpha_1}(n')=n'$ (see \eqref{rootdec} and \eqref{rootde2}).
\end{Rem}

Let $\GAAA=[\GAA,\GAA]= \GAAA_{+}\oplus\GAAA_{-}\oplus\mathbb{C}c$ be the principal Heisenberg subalgebra
of $\GE$, where
\begin{align*}
\GAA = \GAAA_{+}\oplus\GAAA_{-}\oplus\mathbb{C}c\oplus\mathbb{C}d,\quad
\GAAA_{\pm}=\bigoplus_{\pm n>0}\HPPRIN_{(n)}\otimes t^n.
\end{align*}

\begin{Rem}[{\cite[Proposition 5.4.(1)]{LW3}}]\label{indv}
The induced $\GAA$-module 
\begin{align*}
V=\IND_{\GAAA_{+}\oplus\mathbb{C}c\oplus\mathbb{C}d}^{\GAA}\mathbb{C}
\cong U(\GAAA_-)=\mathbb{C}[B(n)\mid n\in\mathbb{Z}_{<0}\setminus3\mathbb{Z}].
\end{align*}
is irreducible (even as an $\GAAA$-module), where $\GAAA_{+}\oplus\mathbb{C}d$ (resp. $c$) acts trivially (resp. as 1) on $\mathbb{C}$.
\end{Rem}


In this paper, for a vector space $U$, 
we denote by $U\{\zeta\}$ (resp. $U\{\zeta_1,\zeta_2\}$) the set of formal power series $\sum_{n\in\mathbb{Z}}u_n\zeta^n$ 
(resp. $\sum_{n,n'\in\mathbb{Z}}u_{n,n'}\zeta_1^n\zeta_2^{n'}$),
where $u_n\in U$ (resp. $u_{n,n'}\in U$)~\cite[\S2]{LW3}. 
For $M\in\mathfrak{sl}_3$, we put (see \eqref{emurecall})
\begin{align*}
M(\zeta)=\sum_{n\in\mathbb{Z}}M(n)\zeta^n\in\GE\{\zeta\}.
\end{align*}

\begin{Rem}\label{koukan}
By ~\cite[(2.44), Theorem 2.4]{LW3},
the commutation relations are given in terms of the generating functions in $\GE\{\zeta_1,\zeta_2\}$
and $\GE\{\zeta\}$ as follows. 
\begin{align*}
[x_{\alpha}(\zeta_1),x_{\beta}(\zeta_2)]
&=
\frac{1}{m}\sum_{s\in C_{-1}} \varepsilon(\nu^s\alpha,\beta)x_{\nu^s\alpha+\beta}(\zeta_2)
\delta(\omega^{-s}\zeta_1/\zeta_2)\\
&\quad +
\frac{\langle x_{\beta},x_{-\beta}\rangle}{m^2}
\sum_{s\in C_{-2}} 
\left(cD\delta(\omega^{-s}\zeta_1/\zeta_2)-m\beta(\zeta_2)\delta(\omega^{-s}\zeta_1/\zeta_2)\right),\\
[\gamma(\zeta_1),x_{\beta}(\zeta_2)] 
&=
\frac{1}{m}\sum_{s\in I}\langle\nu^s\alpha,\beta\rangle
x_{\beta}(\zeta_2)\delta(\omega^{-s}\zeta_1/\zeta_2),\\
x_{\nu^p\alpha}(\zeta) &= x_{\alpha}(\omega^p\zeta),
\end{align*}
where $\delta(\zeta)=\sum_{n\in\mathbb{Z}}\zeta^n$, $D\delta(\zeta)=\sum_{n\in\mathbb{Z}}n\zeta^n$, $\alpha,\beta\in\Phi$,  $\gamma\in\HPPRIN$, $p\in I$,
$C_k=\{s\in I\mid \langle \nu^s\alpha,\beta\rangle=k\}$, and
$\varepsilon(\alpha',\beta')\in\mathbb{C}^{\times}$
is defined by $[x_{\alpha'},x_{\beta'}]=\varepsilon(\alpha',\beta')x_{\alpha'+\beta'}$
if $\alpha',\beta',\alpha'+\beta'\in\Phi$ (see ~\cite[(2.34)]{LW3}).
\end{Rem}

For $s\in\mathbb{Z}$ and a Lie subalgebra $\mathfrak{r}\subseteq \GE$, we denote by $U(\mathfrak{r})_{s}$ the subspace of principal degree $s$ elements in $U(\mathfrak{r})$.
For integers $\ell\geq 0$ and $b$, we define a subspace 
\begin{align*}
  \Theta_{\ell,b} = \sum_{\substack{0\leq\ell'\leq\ell, \beta_1,\dots,\beta_{\ell'}\in\{\pm\alpha_1\}, \\ s,t,m_1,\dots,m_{\ell'}\in\mathbb{Z}, s+m_1+\dots+m_{\ell'}+t=b}}
U(\GAAA_-)_{s}
x_{\beta_1}(m_1)\dots x_{\beta_{\ell'}}(m_{\ell'})
U(\GAAA_+\oplus\mathbb{C}c\oplus\mathbb{C}d)_{t}
\end{align*}
of the universal enveloping algebra $U(\GE)=\bigcup_{\ell\geq 0,b\in\mathbb{Z}}\Theta_{\ell,b}$. Note that $\Theta_{0,b}=U(\GAA)_{b}$. We also define $\Theta_{-1,b}=\{0\}$ for convenience.

\begin{Rem}[{\cite[\S2]{MP}}]\label{sortresRM}
By Remark \ref{koukan}, for $\beta\in\Phi$, $b,m\in\mathbb{Z}$, $\ell\geq -1$, we have
\begin{align}
  x_{\beta}(m)\Theta_{\ell,b}\subseteq
  \Theta_{\ell+1,b+m},\quad
  \Theta_{\ell,b}x_{\beta}(m)\subseteq
  \Theta_{\ell+1,b+m}
  \label{thetaideal}
\end{align}
and $U(\GAA)_m\Theta_{\ell,b}\subseteq\Theta_{\ell,b+m}$, $\Theta_{\ell,b}U(\GAA)_m\subseteq\Theta_{\ell,b+m}$.
It is also easy to see that we have 
\begin{align}
x_{\beta_1}(m_1)\cdots x_{\beta_{\ell}}(m_{\ell})-x_{\beta_{p(1)}}(m_{p(1)})\cdots x_{\beta_{p(\ell)}}(m_{p(\ell)})\in \Theta_{\ell-1,m_1+\dots+m_{\ell}}
\label{sortres}
\end{align}
for $\ell\geq 0$, a permutation $p\in\mathfrak{S}_{\ell}$ and $\beta_1,\dots,\beta_{\ell}\in\{\pm\alpha_1\}$, $m_1,\dots,m_{\ell}\in\mathbb{Z}$ (see ~\cite[(2.11)]{MP}). Thus, we have
\begin{align*}
  \Theta_{\ell,b} = \sum_{\substack{0\leq\ell'\leq\ell, \beta_1,\dots,\beta_{\ell'}\in\{\pm\alpha_1\}, \\ s,t,m_1,\dots,m_{\ell'}\in\mathbb{Z}, s+m_1+\dots+m_{\ell'}+t=b,\\ m_1\leq\dots\leq m_{\ell'} \\ \textrm{$m_i=m_j,\beta_i\ne\beta_j$ implies $\beta_i=\alpha_1, \beta_j=-\alpha_1$ for $i<j$}}}
U(\GAAA_-)_{s}
x_{\beta_1}(m_1)\dots x_{\beta_{\ell'}}(m_{\ell'})
U(\GAAA_+\oplus\mathbb{C}c\oplus\mathbb{C}d)_{t}
\end{align*}
\end{Rem}

\subsection{The principal realization of the basic module}\label{vor}
In virtue of ~\cite{KKLW} (see also \cite[Theorem 8.7]{LW3}),
the assignment $x_{\beta}(n')\mapsto X(\beta,n')$ by
\begin{align}
\begin{split}
X(\beta,\zeta) &= \sum_{n'\in\mathbb{Z}}X(\beta,n')\zeta^{n'}  = \Lambda_0(\pi_{(0)}(x_{\beta})) E^-(-\beta,\zeta)E^+(-\beta,\zeta), \\
E^{\pm}(\beta,\zeta) &= \sum_{\pm n\geq 0}E^{\pm}(\beta,n)\zeta^n
=\exp\left(m\sum_{\pm j>0}\frac{\beta(j)}{j}\zeta^j\right),
\end{split}
\label{vvoorr}
\end{align}
where $\beta\in\Phi$ and $n'\in\mathbb{Z}$, 
in addition to the $\GAA$-module structure (see Remark \ref{indv})
identifies $V$ with the basic $\GE$-module $V(\Lambda_0)$ under the isomorphism $\GEE(A^{(1)}_2)\cong \GE$. 
By \eqref{rootde2}, it is not difficult to see (cf. ~\cite[p.104]{KKLW}, where $\varepsilon=\omega$ and $\varepsilon^i/(\varepsilon^i-1)=(1-\omega^i)/3$)
\begin{align}
\begin{split}
\Lambda_0(\pi_{(0)}(x_{\alpha_1})) = 
\Lambda_0(\pi_{(0)}(x_{\alpha_2})) = 
\Lambda_0(\pi_{(0)}(x_{-(\alpha_1+\alpha_2)})) = \frac{1-\omega}{9},\\
\Lambda_0(\pi_{(0)}(x_{-\alpha_1})) = 
\Lambda_0(\pi_{(0)}(x_{-\alpha_2})) = 
\Lambda_0(\pi_{(0)}(x_{\alpha_1+\alpha_2})) = \frac{1-\omega^2}{9}.
\end{split}
\label{explicitcvalue}
\end{align}


\begin{Lem}[{\cite[Proposition 3.5, Proposition 3.6]{LW3}}]
For $\alpha,\beta\in\Phi$, we have
\begin{align*}
X(\alpha,\zeta_1)E^-(\beta,\zeta_2) &=
E^-(\beta,\zeta_2)X(\alpha,\zeta_1)\myphi_{\alpha,\beta}(\zeta_1/\zeta_2),\\
E^+(\alpha,\zeta_1)X(\beta,\zeta_2) &=
X(\beta,\zeta_2)E^+(\alpha,\zeta_1)\myphi_{\alpha,\beta}(\zeta_1/\zeta_2),
\end{align*}
in $(\END V)\{\zeta_1,\zeta_2\}$, where
\begin{align*}
\myphi_{\alpha,\beta}(x)=\prod_{p\in I}(1-\omega^{-p}x)^{-\langle\nu^p\alpha,\beta\rangle}.
\end{align*}
\end{Lem}


\begin{Ex}\label{exppoly}
The following explicit values will be used in \S\ref{cal}.
\begin{align*}
\myphi_{\alpha_1,\alpha_1}(x) &= \myphi_{-\alpha_1,-\alpha_1}(x)=\frac{1-x^3}{(1-x)^3}=1+\sum_{k\geq 1}3kx^k,\\
\myphi_{-\alpha_1,\alpha_1}(x) &= \myphi_{\alpha_1,-\alpha_1}(x)=\frac{(1-x)^3}{1-x^3}=1+\sum_{k\geq 1}3(x^{3k-1}-x^{3k-2}).
\end{align*}
\end{Ex}

\begin{Cor}[{\cite[Corollary 2.2.12]{Nan}}]\label{idounan}
For $\alpha,\beta\in\Phi$, let $\myphi_{\alpha,\beta}(x)=\sum_{k\geq 0}c_kx^k$. For $n,n'\in\mathbb{Z}$, we have
\begin{align*}
X(\alpha,n)E^-(\beta,n') &= \sum_{k\geq 0}c_kE^{-}(\beta,n'+k)X(\alpha,n-k),\\
E^+(\alpha,n)X(\beta,n') &= \sum_{k\geq 0}c_kX(\beta,n'+k)E^+(\alpha,n-k),
\end{align*}
where $E^-(\beta,n')=0$ (resp. $E^+(\alpha,n)=0$) when $n'>0$ (resp. $n<0$).
More generally, 
for $n_1,\dots,n_{\ell},n'_1,\dots,n'_{\ell}\in\mathbb{Z}$ and $\alpha_1,\dots,\alpha_{\ell},\beta_1,\dots,\beta_{\ell}\in\Phi$, we have
\begin{align*}
{} &{} X(\alpha_1,n_1)\cdots X(\alpha_{\ell},n_{\ell})E^-(\beta,n') \\
&= 
\sum_{j_1,\cdots,j_{\ell}\geq 0}c_{j_1}\cdots c_{j_{\ell}}E^{-}(\beta,n'+j_1+\dots+j_{\ell})X(\alpha_1,n-j_1)\cdots X(\alpha_{\ell},n_{\ell}-j_{\ell}),\\
{} &{} E^+(\alpha,n)X(\beta_1,n'_1)\cdots X(\beta_{\ell},n'_{\ell}) \\
&= 
\sum_{j_1,\cdots,j_{\ell}\geq 0}c_{j_1}\cdots c_{j_{\ell'}}X(\beta_1,n'+j_1)\cdots X(\beta_{\ell},n'+{j_{\ell}}) E^+(\alpha,n-j_1-\cdots-j_{\ell}).
\end{align*}
\end{Cor}

\subsection{The triple tensor product of the basic module}\label{tripletensor}
For $\alpha,\beta\in\Phi$, let 
\begin{align*}
P_{\alpha,\beta}(x)=\prod_{p\in I, \langle\nu^p\alpha,\beta\rangle<0}(1-\omega^{-p}x)^{-\langle\nu^p\alpha,\beta\rangle}.
\end{align*}

\begin{Ex}\label{exppoly2}
The following explicit values will be used in \S\ref{cal}.
\begin{align*}
P_{\alpha_1,\alpha_1}(x) &= P_{-\alpha_1,-\alpha_1}(x)=1+x+x^2,\\
P_{\alpha_2,-\alpha_1}(x) &= P_{\alpha_1,\alpha_1+\alpha_2}(x)=(1-\omega x)^2,\\
P_{-(\alpha_1+\alpha_2),-\alpha_1}(x) &= P_{\alpha_1,-\alpha_2}(x)=(1-\omega^2 x)^2.
\end{align*}
\end{Ex}

Let $W$ be a $\GE$-module in a category $\mathcal{C}_K$ for $K\in\mathbb{C}$ (see ~\cite[\S3]{LW3}).
By Remark \ref{koukan},
\begin{align}
\lim_{\zeta_1,\zeta_2\to\zeta}P_{\alpha,\beta}(\zeta_1/\zeta_2)x_{\alpha}(\zeta_1)x_{\beta}(\zeta_2)
\label{limdef}
\end{align}
makes sense as an element of $(\END W)\{\zeta\}$ (see ~\cite[\S4,\S5]{MP}), and we denote it by $x_{\alpha,\beta}(\zeta)$.
In the rest of this paper, we take $W=V^{\otimes 3}(\cong V(\Lambda_0)^{\otimes 3})$.

\begin{Prop}\label{mpprop}
We have the following relations in $(\END W)\{\zeta\}$.
\begin{enumerate}
\item\label{mpprop1} $x_{-\alpha_1,-\alpha_1}(\zeta)=
\frac{2\Lambda_0(\pi_{(0)}(x_{-\alpha_1}))^{2}P_{-\alpha_1,-\alpha_1}(1)}{\Lambda_0(\pi_{(0)}(x_{\alpha_1}))}
  E^{-}(\alpha_1,\zeta)x_{\alpha_1}(\zeta)E^{+}(\alpha_1,\zeta)$.
\item\label{mpprop2} $x_{\alpha_1,\alpha_1}(\zeta)=
\frac{2\Lambda_0(\pi_{(0)}(x_{\alpha_1}))^{2}P_{\alpha_1,\alpha_1}(1)}{\Lambda_0(\pi_{(0)}(x_{-\alpha_1}))}
  E^{-}(-\alpha_1,\zeta)x_{-\alpha_1}(\zeta)E^{+}(-\alpha_1,\zeta)$.
\item\label{mpprop3} $x_{\alpha_2,-\alpha_1}(\zeta)=
E^{-}(\alpha_1,\zeta)x_{\alpha_1,\alpha_1+\alpha_2}(\zeta)E^{+}(\alpha_1,\zeta)$.
\item\label{mpprop4} $x_{-(\alpha_1+\alpha_2),-\alpha_1}(\zeta)=
E^{-}(\alpha_1,\zeta)x_{\alpha_1,-\alpha_2}(\zeta)E^{+}(\alpha_1,\zeta)$.
\end{enumerate}
\end{Prop}

\begin{proof}
These are similarly established as the proof of ~\cite[Theorem 6.6]{MP}.
For explicit values of the constants, see \eqref{explicitcvalue} and Example \ref{exppoly2}.
\end{proof}

\begin{Prop}\label{high}
For $a_1,a_2,a_3\in\mathbb{C}$, we define a vector in $W$ by
\begin{align*}
u_{a_1,a_2,a_3} = a_1(B(-1)\otimes 1\otimes 1) + a_2(1\otimes B(-1)\otimes 1) + a_3 (1\otimes 1\otimes B(-1)).
\end{align*}
If $a_1+a_2+a_3=0$ and $(a_1,a_2,a_3)\ne (0,0,0)$, then $u_{a_1,a_2,a_3}$ is a highest weight vector
with highest weight $3\Lambda_0-\alpha_0=\Lambda_0+\Lambda_1+\Lambda_2-\delta$.
\end{Prop}

\begin{proof}
It is easily verified by
direct calculation using $B(-1)=f_0+f_1+f_2$ in $\GE$ (see ~\S\ref{prafflie}) and
$(f_0+f_1+f_2)1=f_{0}1$ in $V\cong V(\Lambda_{0})$ (see ~\cite[(10.1.1)]{Kac}).
\end{proof}

\begin{Rem}\label{notsub}
The level 3 module $V(2\Lambda_i+\Lambda_j+n\delta)$, where $i\ne j\in I$ and $n\in\mathbb{Z}$,
is not a submodule of $W$ because the equation in the weight lattice
\begin{align*}
3\Lambda_{0}-(k_0\alpha_0+k_1\alpha_1+k_2\alpha_2)\equiv  p_0\Lambda_0+p_1\Lambda_1+p_2\Lambda_2\pmod{\mathbb{Z}\delta},
\end{align*}
where $k_0,k_1,k_2,p_0,p_1,p_2\in\mathbb{Z}$,
implies $p_a\equiv p_b\pmod{3}$ for $a\ne b\in I$.
This is easily seen by the fact that any entry in $A^{(1)}_2$ is congruent to 2
modulo 3 (see \eqref{A12GCM}).
\end{Rem}

\section{Spanning vectors}\label{maincomp}

\subsection{Colored integers}
Let $\CZ=\{\cint{n}\mid n\in\mathbb{Z}\}$ be the set of colored integers
and put $\AZ=\mathbb{Z}\sqcup\CZ$. 
Recall \eqref{orde} in \S\ref{mainse}.
There, we defined the order $\geq$ for positive integers and colored positive integers for simplicity.
The order $\geq$ below is a generalization of the order $\geq$ in \S\ref{mainse}.

\begin{Def}\label{twoorders}
On the set $\AZ$, we consider total orders $\geq$ and $\ANOE$ such that
\begin{align*}
\dots>2>\cint{2}>1>\cint{1}>0>\cint{0}>-1>\cint{-1}>-2>\cint{-2}>\cdots, \\
\cdots\ANO\cint{2}\ANO 2\ANO\cint{1}\ANO 1\ANO\cint{0}\ANO 0\ANO\cint{-1}\ANO -1\ANO\cint{-2}\ANO -2\ANO\cdots.
\end{align*}
\end{Def}


Similarly, we (re)define $\COLOR(n)\in\{\pm\}$ and $\CONT(n)\in\mathbb{Z}$ for $n\in\AZ$ by
\begin{align*}
\COLOR(n) = \begin{cases}
+ & \textrm{if $n\in\mathbb{Z}$},\\
- & \textrm{if $n\in\CZ$},
\end{cases}\quad
\CONT(n) = \begin{cases}
n & \textrm{if $n\in\mathbb{Z}$},\\
m & \textrm{if $n=\cint{m}$ for some $m\in\mathbb{Z}$}.
\end{cases}
\end{align*}

Let $\ZSEQ$ be the set of finite length sequences $\boldsymbol{n}=(n_1,\dots,n_{\ell})$
of $\AZ$. 
We put $\SHAPE(\boldsymbol{n})=(\CONT(n_1),\dots,\CONT(n_{\ell}))$, and (re)define the length $\LENGTH(\boldsymbol{n})$ and the size $\SUM{\boldsymbol{n}}$ of $\boldsymbol{n}$ by
\begin{align*}
\LENGTH(\boldsymbol{n})=\ell,\quad
\SUM{\boldsymbol{n}}=\CONT(n_1)+\dots+\CONT(n_{\ell}).
\end{align*}

\begin{Ex}\label{ex2color}
For $\boldsymbol{n}=(\cint{-5},-5,\cint{-5},-5,\cint{-6})$, we have
$\LENGTH(\boldsymbol{n})=5$, $\SUM{\boldsymbol{n}}=-26$ and $\SHAPE(\boldsymbol{n})=(-5,-5,-5,-5,-6)$.
\end{Ex}


\subsection{Lexicographical orders}
Let $X$ be a set with a total order $\AANO$.
The result $(n'_1,\dots,n'_{\ell})$ of sorting
a finite sequence $\boldsymbol{n}=(n_1,\dots,n_{\ell})\in X^{\ell}$
so that $n'_1\AANO\cdots\AANO n'_{\ell}$ is denoted by $\SORT^{(X,\AANO)}(\boldsymbol{n})$, where $\ell\geq 0$. For example, 
we have
$\SORT^{(\AZ,\ANOO)}(\boldsymbol{n})=(\cint{-6},-5,-5,\cint{-5},\cint{-5})$
in Example \ref{ex2color}.

\begin{Def}\label{lexdef3}
For two finite sequences $\boldsymbol{m}=(m_1,\dots,m_{\ell})$ and $\boldsymbol{n}=(n_1,\dots,n_{\ell})$ of $X$ that have the same length $\ell\geq 0$,
we write $\boldsymbol{m}\LEXO{(X,\AANO)}\boldsymbol{n}$ if there exists
$1\leq j\leq\ell$ such that $n_j\NAANO m_j$ 
and $m_k=n_k$ for $k<j$.
\end{Def}

The following four lemmata are easily proved, and we omit proofs.

\begin{Lem}\label{transx}
The binary relation $\LEXO{(X,\AANO)}$ (on $X^{\ell}$ for $\ell\geq 0$) is transitive (and thus $\LEXXO{(X,\AANO)}$ is a total order).
\end{Lem}

\begin{Lem}\label{sortx}
For $\ell\geq 0$ and $\boldsymbol{n}\in X^{\ell}$, we have
$\boldsymbol{n}\LEXXO{(X,\AANO)}\SORT^{(X,\AANO)}(\boldsymbol{n})$.
\end{Lem}

\begin{Lem}\label{mergex}
Let $\boldsymbol{m},\boldsymbol{n}\in X^{\ell}$, where $\ell\geq 0$.
If $\boldsymbol{m}\LEXO{(X,\AANO)}\boldsymbol{n}$ and $\boldsymbol{m}=\SORT^{(X,\AANO)}(\boldsymbol{m})$,
then we have $\SORT^{(X,\AANO)}((\boldsymbol{h},\boldsymbol{m},\boldsymbol{t}))\LEXO{(X,\AANO)}\SORT^{(X,\AANO)}((\boldsymbol{h},\boldsymbol{n},\boldsymbol{t}))$
for any finite sequences $\boldsymbol{h},\boldsymbol{t}$ of $X$.
Here, $(\boldsymbol{h},\boldsymbol{m},\boldsymbol{t})$ stands for the concatenation of $\boldsymbol{h},\boldsymbol{m}$ and $\boldsymbol{t}$.
\end{Lem}

For example, $(\boldsymbol{h},\boldsymbol{m},\boldsymbol{t})=(\cint{-6},\cint{-5},\cint{-6},-6,\cint{-5},\cint{-5})$ when $\boldsymbol{h}=(\cint{-6},\cint{-5})$, $\boldsymbol{m}=(\cint{-6})$ and $\boldsymbol{t}=(-6,\cint{-5},\cint{-5})$ when $X=\AZ$.



\begin{Lem}\label{lexcomp3x}
Let $\boldsymbol{m},\boldsymbol{n}$
be two finite sequences of $X$ such that
$(\boldsymbol{m},\boldsymbol{n})=\SORT^{(X,\AANO)}((\boldsymbol{m},\boldsymbol{n}))$.
We have $(\boldsymbol{m},\boldsymbol{n})\LEXO{(X,\AANO)}\SORT^{(X,\AANO)}((\boldsymbol{m}',\boldsymbol{n}'))$ for any finite sequences $\boldsymbol{m}',\boldsymbol{n}'$ of $X$ such that
$\LENGTH(\boldsymbol{m})=\LENGTH(\boldsymbol{m}')$,
$\LENGTH(\boldsymbol{n})=\LENGTH(\boldsymbol{n}')$ and
$\boldsymbol{m}\LEXO{(X,\AANO)}\boldsymbol{m}'$.
\end{Lem}


\begin{Def}\label{lexdef}
For $\boldsymbol{m}=(m_1,\dots,m_{\ell}),\boldsymbol{n}=(n_1,\dots,n_{\ell})\in\ZSEQ$, we write $\boldsymbol{m}\LEX\boldsymbol{n}$
if we have either $\SHAPE(\boldsymbol{m})\LEXO{(\mathbb{Z},\leq)}\SHAPE(\boldsymbol{n})$ (evidently, exclusive) or
$\SHAPE(\boldsymbol{m})=\SHAPE(\boldsymbol{n})$ and
$\boldsymbol{m}\LEXO{(\AZ,\ANOO)}\boldsymbol{n}$.
\end{Def}

Note that $\LEX$ is transitive, which follows from Lemma \ref{transx}.
Thus, $\LEXX$ is a total order (on $\AZ^{\ell}$ for $\ell\geq 0$).

\begin{Rem}\label{orderimportant}
We emphasize that our definition of $\boldsymbol{m}\LEX\boldsymbol{n}$, where $\boldsymbol{m},\boldsymbol{n}\in\ZSEQ$, 
requires $\LENGTH(\boldsymbol{m})=\LENGTH(\boldsymbol{n})$
while we do not require $\SUM{\boldsymbol{m}}=\SUM{\boldsymbol{n}}$.
\end{Rem}

\begin{Ex}
We have $(-5,-3,\cint{-3})\LEX (\cint{-5},-4,1)$ and $(-7,-5,\cint{-5})\LEX (-7,-5,-5)$.
\end{Ex}

Throughout this paper, we put $\SORT=\SORT^{(\AZ,\ANOO)}$.
Note that we have 
\begin{align}
\SHAPE(\SORT(\boldsymbol{n}))=\SORT^{(\mathbb{Z},\leq)}(\SHAPE(\boldsymbol{n}))
\label{shrel}
\end{align}
for $\boldsymbol{n}\in\ZSEQ$ by the fact that $a\ANOO b$ implies $\CONT(a)\leq\CONT(b)$.

\begin{Cor}\label{lexcomp2}
For $\boldsymbol{n}\in\ZSEQ$, we have $\boldsymbol{n}\LEXX\SORT(\boldsymbol{n})$.
\end{Cor}

\begin{proof}
It directly follows from \eqref{shrel} and Lemma \ref{sortx}. 
\end{proof}

\begin{Cor}\label{lexcomp4}
Let $\boldsymbol{n}=(n_1,\dots,n_{\ell})\in\ZSEQ$.
We have
\begin{align*}
\boldsymbol{n}\LEX\SORT((n_1-\varepsilon_1,\dots,n_{\ell}-\varepsilon_{\ell}))
\end{align*}
for nonnegative integers $\varepsilon_1,\dots,\varepsilon_{\ell}\geq 0$ such that
$\varepsilon_1+\dots+\varepsilon_{\ell}\geq 1$.
Here, for $n\in\AZ$ and $\varepsilon\in\mathbb{Z}$, we define $m=n-\varepsilon\in\AZ$ so that $\CONT(m)=\CONT(n)-\varepsilon$ and $\COLOR(m)=\COLOR(n)$.
\end{Cor}

\begin{proof}
It is not difficult to show $\SHAPE(\boldsymbol{n})\LEXO{(\mathbb{Z},\leq)}\SHAPE(\SORT((n_1-\varepsilon_1,\dots,n_{\ell}-\varepsilon_{\ell})))$ by \eqref{shrel}, Lemma \ref{transx} and Lemma \ref{sortx}.
\end{proof}

The following two lemmata are easily deduced by \eqref{shrel} and Lemma \ref{mergex}, Lemma \ref{lexcomp3x} respectively.

\begin{Cor}\label{lexcomp}
Let $\boldsymbol{m},\boldsymbol{n}\in\ZSEQ$.
If $\boldsymbol{m}\LEX\boldsymbol{n}$ and $\boldsymbol{m}=\SORT(\boldsymbol{m})$,
then we have $\SORT((\boldsymbol{h},\boldsymbol{m},\boldsymbol{t}))\LEX\SORT((\boldsymbol{h},\boldsymbol{n},\boldsymbol{t}))$
for $\boldsymbol{h},\boldsymbol{t}\in\ZSEQ$.
\end{Cor}

\begin{Cor}\label{lexcomp3}
Let $\boldsymbol{m},\boldsymbol{n}\in\ZSEQ$
be
$(\boldsymbol{m},\boldsymbol{n})=\SORT((\boldsymbol{m},\boldsymbol{n}))$.
We have $(\boldsymbol{m},\boldsymbol{n})\LEX\SORT((\boldsymbol{m}',\boldsymbol{n}'))$ for $\boldsymbol{m}',\boldsymbol{n}'\in\ZSEQ$ such that
$\LENGTH(\boldsymbol{m})=\LENGTH(\boldsymbol{m}')$,
$\LENGTH(\boldsymbol{n})=\LENGTH(\boldsymbol{n}')$,
$\boldsymbol{m}\LEX\boldsymbol{m}'$, $\SHAPE(\boldsymbol{m})\ne\SHAPE(\boldsymbol{m}')$.
\end{Cor}

In \S\ref{cal}, we apply Corollary \ref{lexcomp3} under $\SUM{\boldsymbol{m}}\ne\SUM{\boldsymbol{m}'}$, which implies $\SHAPE(\boldsymbol{m})\ne\SHAPE(\boldsymbol{m}')$.





\subsection{Reducibilities}\label{redchap}
In the rest, 
we put $x_{\pm\alpha_1}(n)=\YMP{n}$ for $n\in\mathbb{Z}$ and
$\YY(\boldsymbol{n})=\YY_{\COLOR(n_1)}(\CONT(n_1))\cdots\YY_{\COLOR(n_{\ell})}(\CONT(n_{\ell}))$ for $\boldsymbol{n}\in\ZSEQ$. 


\begin{Def}\label{defnseq}
Let $\NSEQ$ be the subset of $\ZSEQ$ consisting of $\boldsymbol{n}=(n_1,\dots,n_{\ell})\in\ZSEQ$
such that $\boldsymbol{n}=\SORT(\boldsymbol{n})$ (i.e., $\boldsymbol{n}$ is weakly increasing
by the order $\ANOO$) and $\CONT(n_i)<0$ for $1\leq i\leq\ell$ (i.e., $\boldsymbol{n}$ consists of
negative integers or colored negative integers).
\end{Def}

Note that the assignment $(\lambda_1,\dots,\lambda_{\ell})\mapsto (-\lambda_1,\dots,-\lambda_{\ell})$ is a bijection from the set of 2-colored partitions (in the sense of \S\ref{mainse}) to $\NSEQ$.


\begin{Lem}
\label{sortres3}
For $\ell\geq 0$, $s\in\mathbb{Z}$ and a highest weight vector $\VZ\in W$, we have 
\begin{align*}
\Theta_{\ell,s}\VZ
=\sum_{\substack{\boldsymbol{m}\in\NSEQ \\ \LENGTH(\boldsymbol{m})\leq\ell, s\leq\SUM{\boldsymbol{m}}}} 
U(\GAAA_-)_{s-\SUM{\boldsymbol{m}}}\BUI{\boldsymbol{m}}.
\end{align*}
\end{Lem}

\begin{proof}
It follows from Remark \ref{sortresRM},
highestness of $\VZ$ and $U(\GAAA_-)_{t}=\{0\}$ for $t>0$.
\end{proof}

\begin{Def}
Let $\VZ\in W$ be a highest weight vector. 
We say that an element $\boldsymbol{n}$ in $\ZSEQ$ is $\VZ$-reducible if $\BUI{\boldsymbol{n}}\in W_{>\boldsymbol{n}}(\VZ)$, where (see also Remark \ref{orderimportant})
\begin{align*}
W_{>\boldsymbol{n}}(\VZ)
= 
\sum_{\substack{\boldsymbol{m}\in\NSEQ \\ \LENGTH(\boldsymbol{n})\geq\LENGTH(\boldsymbol{m}),
    \SUM{\boldsymbol{n}}<\SUM{\boldsymbol{m}}}} 
U(\GAAA_-)_{\SUM{\boldsymbol{n}}-\SUM{\boldsymbol{m}}}\BUI{\boldsymbol{m}}
+ \sum_{\substack{\boldsymbol{m}\in\NSEQ \\ \LENGTH(\boldsymbol{n})>\LENGTH(\boldsymbol{m}), \SUM{\boldsymbol{n}}=\SUM{\boldsymbol{m}}}} \mathbb{C}\BUI{\boldsymbol{m}}
+ \sum_{\substack{\boldsymbol{m}\in\NSEQ \\ \boldsymbol{n}\LEX\boldsymbol{m}, \SUM{\boldsymbol{n}}=\SUM{\boldsymbol{m}}}} \mathbb{C}\BUI{\boldsymbol{m}}.
\end{align*}
Otherwise, we say that $\boldsymbol{n}$ is $\VZ$-irreducible.
\end{Def}


\begin{Rem}\label{deghomgpart2}
By Lemma \ref{sortres3}, for $\boldsymbol{n}\in\ZSEQ$, we have
\begin{align*}
\Theta_{\LENGTH(\boldsymbol{n})-1,\SUM{\boldsymbol{n}}}\VZ\subseteq
W_{>\boldsymbol{n}}(\VZ).
\end{align*}
\end{Rem}

\begin{Prop}\label{ensort}
For $\boldsymbol{n}\in\ZSEQ\setminus\NSEQ$, 
$\boldsymbol{n}$ is $\VZ$-irreducible.
\end{Prop}

\begin{proof}
Let $\boldsymbol{n}=(n_1,\dots,n_{\ell})$ and $\boldsymbol{n}'=\SORT(\boldsymbol{n})$. By \eqref{sortres}, we have $\BUI{\boldsymbol{n}}-\BUI{\boldsymbol{n}'}\in\Theta_{\LENGTH(\boldsymbol{n})-1,\SUM{\boldsymbol{n}}}\VZ$.
If $\CONT(n_i)>0$ for some $i$,
we have $\BUI{\boldsymbol{n}'}=0$.
If otherwise and
$\CONT(n_j)= 0$ for some $j$, 
we have $\BUI{\boldsymbol{n}'}\in\Theta_{\LENGTH(\boldsymbol{n})-1,\SUM{\boldsymbol{n}}}\VZ$.
If otherwise, we have $\boldsymbol{n}\ne\boldsymbol{n}'$ and thus
we have $\boldsymbol{n}\LEX\boldsymbol{n}'$ by Corollary \ref{lexcomp2}.
\end{proof}


\begin{Prop}\label{spaneq}
Let $\RED(\VZ)$ be the subset of $\NSEQ$, consisting of all $\VZ$-irreducible elements $\boldsymbol{n}$. We have 
\begin{align*}
U(\GE)\VZ = \sum_{\boldsymbol{n}\in\RED(\VZ)} U(\GAAA_-)\BUI{\boldsymbol{n}}.
\end{align*}
\end{Prop}

\begin{proof}
It is routine to show
$\BUI{\boldsymbol{m}}\in\sum_{\boldsymbol{n}\in\RED(\VZ)} U(\GAAA_-)\BUI{\boldsymbol{n}}$ for any $\boldsymbol{m}\in\NSEQ$
by induction on $-\SUM{\boldsymbol{m}}$ (and on $\LENGTH(\boldsymbol{m})$ and on $\LEXX$).
\end{proof}

\subsection{Forbidden patterns}
We say that $\boldsymbol{n}=(n_1,\dots,n_{\ell})\in\ZSEQ$ contains
$\boldsymbol{m}=(m_1,\dots,m_{\ell'})\in\ZSEQ$ if there exists $0\leq i\leq \ell-\ell'$ such that
$n_{j+i}=m_j$ for $1\leq j\leq \ell'$.

\begin{Def}
We say that $\boldsymbol{m}\in\ZSEQ$ is a forbidden pattern
if $\boldsymbol{n}\in\ZSEQ$ is $\VZ$-reducible for any highest weight vector $\VZ\in W$
whenever $\boldsymbol{n}$ contains $\boldsymbol{m}$.
\end{Def}

\begin{Thm}\label{forthm}
For $k\in\mathbb{Z}$, the following elements in $\ZSEQ$ 
are forbidden patterns.
\begin{enumerate}
\item[(1)] $(-k,-k), (\cint{-k},\cint{-k}), (-k-1,-k), (\cint{-k-1},\cint{-k})$.
\item[(2)] $(\cint{-1-3k},-3k),(-1-3k,\cint{-3k})$.
\item[(3)] $(-1-3k,\cint{-1-3k}),(\cint{-2-3k},-3k)$.
\item[(4)] $(-2-3k,\cint{-2-3k}),(\cint{-3-3k},-1-3k)$.
\item[(5)] $(\cint{-3-3k},-2-3k),(-3-3k,\cint{-2-3k})$.
\item[(6)] $(-3-3k,\cint{-3-3k},\cint{-1-3k}),(-5-3k,-3-3k,\cint{-3-3k})$.
\item[(7)] $(\cint{-5-3k},-4-3k,-2-3k,\cint{-1-3k})$.
\end{enumerate}
\end{Thm}

\begin{proof}
We prove that the elements in (1),(2),(3),(4),(5),(6),(7) are forbidden patterns
in \S\ref{forone}, \S\ref{fortwoA}, \S\ref{fortwoB}, \S\ref{fortwoC}, \S\ref{fortwoD}, \S\ref{forthree}, \S\ref{forfour}, respectively. 
\end{proof}

In the rest of this section, we put $w=1\otimes 1\otimes 1\in W$.

\begin{Prop}\label{initcor}
An element $\lambda=(\lambda_1,\dots,\lambda_{\ell})$ in $\NSEQ$ is $w$-reducible if
$\ell\geq 1$ and $\lambda_{\ell}=-1,\cint{-1},\cint{-2}$.
\end{Prop}

\begin{proof}
By direct calculation, we have
$\BUIII{-1}=-\BUIII{\cint{-1}} \in \mathbb{C}B(-1)w$.
Thus, by Remark \ref{sortresRM}, we have $\BUIII{(\lambda_1,\dots,\lambda_{\ell})}\in\Theta_{\ell-1,|\lambda|}$ if $\lambda_{\ell}=-1,\cint{-1}$.
The argument for $\lambda_{\ell}=\cint{-2}$
is similar because of $\BUIII{-2}-\BUIII{\cint{-2}} \in \mathbb{C}B(-2)w$,
\eqref{sortres} and 
 $(\lambda_1,\dots,\lambda_{\ell-1},\cint{-2})\LEX\SORT((\lambda_1,\dots,\lambda_{\ell-1},-2))$ by Corollary \ref{lexcomp}.
\end{proof}

For a 2-colored partition $\lambda=(\lambda_1,\dots,\lambda_{\ell})$, 
we easily see that the condition that $\lambda$ satisfies (D1)--(D3) in \S\ref{mainse} is equivalent to
the condition that $(-\lambda_1,\dots,-\lambda_{\ell})\in\NSEQ$ does not contain the elements (1)--(7) in Theorem \ref{forthm}.
Concerning Proposition \ref{initcor} and (D4), we may restate Proposition \ref{spaneq} as follows. 

\begin{Cor}\label{biidenintercor}
For $i=1$ (resp. $i=2$), put $v_1=u_{a_1,a_2,a_3}$ with $a_1+a_2+a_3=0$ and $(a_1,a_2,a_3)\ne (0,0,0)$ (resp. $v_2=w$).
The set
\begin{align*}
\{\AOB{-\mu_1}\cdots\AOB{-\mu_{\ell'}}\BUII{(-\lambda_1,\dots,-\lambda_{\ell})}{i}\}
\end{align*}
spans a module that is isomorphic to 
$V((2i-1)\Lambda_0+(2-i)\Lambda_1+(2-i)\Lambda_2-(2-i)\delta)$, where $(\mu_1,\dots,\mu_{\ell'})$ varies in $\REG{3}$
and $(\lambda_1,\dots,\lambda_{\ell})$ varies in $\BIR$ (resp. $\BIRP$).
\end{Cor}

\section{Vertex operator calculations}\label{cal}
Recall $\omega=\exp(2\pi\sqrt{-1}/3)$ in \S\ref{vertset}.
In the following of this section,
we fix an integer $k$ and a highest weight vector $\VZ$ in $W$.
We also recall $\YY(\boldsymbol{n})=\YY_{\COLOR(n_1)}(\CONT(n_1))\cdots\YY_{\COLOR(n_{\ell})}(\CONT(n_{\ell}))$ for $\boldsymbol{n}\in\ZSEQ$ (see \S\ref{redchap}). It is called a $\YY$-monomial in this section.
We remark that the arguments in this section are similar to ~\cite{Nan}.

\subsection{Annihilating elements}
Recall \eqref{vvoorr} and \eqref{limdef}.
We denote by $R_i(n)$ the coefficient of $\zeta^n$ in $R_i(\zeta)$ for $i\in\{1,2,\pm\}$, where 
\begin{align*}
R_{+}(\zeta) &= x_{\alpha_1,\alpha_1}(\zeta), \quad\quad
R_{-}(\zeta) = x_{-\alpha_1,-\alpha_1}(\zeta), \\ 
R_1(\zeta) &= E^{-}(-\alpha_1,\zeta)x_{\alpha_2,-\alpha_1}(\zeta)-x_{\alpha_1,\alpha_1+\alpha_2}(\zeta)E^{+}(\alpha_1,\zeta),\\
R_2(\zeta) &= E^{-}(-\alpha_1,\zeta)x_{-(\alpha_1+\alpha_2),-\alpha_1}(\zeta)-x_{\alpha_1,-\alpha_2}(\zeta)E^{+}(\alpha_1,\zeta).
\end{align*}

We abbreviate $E^+(\alpha_1,n)$ (resp. $E^-(-\alpha_1,n)$) to $E^+(n)$ (resp. $E^-(n)$)
for $n\in\mathbb{Z}$. Note that $E^\pm(n)=0$ when $\mp n>0$.

\subsection{Preparations}\label{nandiprep}
For an integer $\ell\geq 0$ and summable expressions (see ~\cite[\S4]{MP}) $G$ and $H$ that have the same principal degree $a\in\mathbb{Z}$,
the notation $G\equiv_{\ell} H$ means that we have $G\VU-H\VU\in\Theta_{\ell,a}\VU$
for any (i.e., not necessarily highest weight) vector $\VU$ in any $\GE$-module
in a category $\mathcal{C}_K$ for any $K\in\mathbb{C}$ (see ~\cite[\S3]{LW3}).

\begin{Def}
Let $\ell\geq 1$ and $a\in\mathbb{Z}$.
In this paper, a good expression of width $\ell$ and degree $a$
is an expansion (i.e., possibly an infinite formal sum) of the form 
\begin{align}
\xi=\sum_{\boldsymbol{p}=(p_1,\dots,p_{\ell})\in\AZ^{\ell}, \SUM{\boldsymbol{p}}=a} c_{\boldsymbol{p}}\YY(\boldsymbol{p}),
\label{goodexp}
\end{align}
where $c_{\boldsymbol{p}}$ are complex numbers,
such that $\SUPP_i(\xi)$ is finite for any $i\in\mathbb{Z}$.
\end{Def}

As usual (see ~\cite[(4.8),(6.19)]{LW3}), $\SUPP_i(\xi)$ is defined to be the set
\begin{align*}
\{\boldsymbol{p}=(p_1,\dots,p_{\ell})\in\AZ^{\ell}\mid c_{\boldsymbol{p}}\ne 0 \textrm{ and }
  \CONT(p_{\ell'})+\dots+\CONT(p_{\ell})\leq i \textrm{ for } 1\leq \ell'\leq\ell\}.
\end{align*}

A good expression stands for a good expression of width $\ell$ and degree $a$
for some $\ell\geq 1$ and $a\in\mathbb{Z}$.
It is clear that a good expression is summable
for any $\GE$-module
in a category $\mathcal{C}_K$ for any $K\in\mathbb{C}$.

We say that a good expression $\xi$ is zero if it is formally zero (i.e.,
$c_{\boldsymbol{p}}=0$ for all $\boldsymbol{p}$).
For a nonzero good expression $\xi$, we define
the leading monomial $\LM(\xi)$ to be the $\YY$-monomial
$\YY(\boldsymbol{p})$ such that
$c_{\boldsymbol{p}}\ne 0$, $\boldsymbol{p}=\SORT(\boldsymbol{p})$ and $\boldsymbol{p}$ is minimum with respect to $\LEXX$ among such in \eqref{goodexp} if exists
(otherwise, $\LM(\xi)$ is not defined).

\begin{Ex}\label{firstexpansionex}
By \eqref{sortres} and
$P_{\beta,\beta}(1)=3$ for $\beta=\pm\alpha_1$ (see Example \ref{exppoly2}), we have 
\begin{align*}
R_{\pm}(-2k)/3 &\equiv_1 \YMP{-k}\YMP{-k}+2\YMP{-k-1}\YMP{-k+1}+2\YMP{-k-2}\YMP{-k+2}+\cdots, \\
R_{\pm}(-2k-1)/3 &\equiv_1 2(\YMP{-k-1}\YMP{-k}+\YMP{-k-2}\YMP{-k+1}+\YMP{-k-3}\YMP{-k+2}+\cdots),
\end{align*}
where the right hand sides are good expressions of width 2 and degree $-2k$, $-2k-1$ with the leading monomials $\YMP{-k}\YMP{-k}$, $\YMP{-k-1}\YMP{-k}$, respectively. Note 
that $\YY$-monomials are arranged in order of $\LEX$.
\end{Ex}


\begin{Prop}\label{forbhantei}
Let $b\geq 1$ and $a\in\mathbb{Z}$.
Assume that we have an expression
\begin{align*}
R = r + \sum_{i\geq 1}E^{-}(-i)r_{i}^{-} + \sum_{i\geq 1}r_{i}^{+}E^{+}(i),
\end{align*}
such that $R\equiv_{b-1}0$, where $r$ is a nonzero good expression of width $b$ and degree $a$ with $\boldsymbol{g}=\LM(r)$,
$r_i^{\pm}$ is a good expression of width $b$ and degree $a\mp i$. If
any $\YY$-monomial $\YY(\boldsymbol{m})$ in $r_i^{+}$ satisfies $\boldsymbol{g}\LEX\boldsymbol{m}$ for $i\geq 1$,
then $\boldsymbol{g}$ is a forbidden pattern.
\end{Prop}

\begin{proof}
Take $\boldsymbol{j},\boldsymbol{h}\in\ZSEQ$ and let $\boldsymbol{n}=(\boldsymbol{j},\boldsymbol{g},\boldsymbol{h})$.
By $R\equiv_{b-1}0$ and \eqref{thetaideal}, we have $\YY(\boldsymbol{j})R\BUI{\boldsymbol{h}}\in\Theta_{\ell(\boldsymbol{n})-1,\SUM{\boldsymbol{n}}}\VZ$,
which is a subspace of $W_{>\boldsymbol{n}}(\VZ)$
by Remark \ref{deghomgpart2}.
By Proposition \ref{ensort},
it is enough to show that $\boldsymbol{n}$ is $\VZ$-reducible assuming
$\boldsymbol{n}=\SORT(\boldsymbol{n})$. 

Let $c_{\boldsymbol{g}}$ be the (nonzero) coefficient of $\YY(\boldsymbol{g})$ in $r$ and
\begin{align*}
U=\sum_{{\boldsymbol{m}\in\NSEQ, \boldsymbol{n}\LEX\boldsymbol{m}, \SUM{\boldsymbol{n}}=\SUM{\boldsymbol{m}}}} \mathbb{C}\BUI{\boldsymbol{m}}.
\end{align*}

By Corollary \ref{lexcomp2}, Corollary \ref{lexcomp} and \eqref{sortres}, we have
\begin{align*}
  \YY(\boldsymbol{j})r\BUI{\boldsymbol{h}}
  -
  c_{\boldsymbol{g}}\BUI{\boldsymbol{n}}
  \in \Theta_{\LENGTH(\boldsymbol{n})-1,\SUM{\boldsymbol{n}}}\VZ
+ U.
\end{align*}

Let $\boldsymbol{h}=(h_1,\dots,h_u)$. By Corollary \ref{idounan}, 
we have
\begin{align*}
  E^{+}(i)\BUI{\boldsymbol{h}}\in\sum_{(i_1,\dots,i_u)\in\mathbb{Z}^u_{\geq 0}, i_1+\dots+i_u=i}\mathbb{C}\BUI{(h_1+i_1,\dots,h_u+i_u)}.
  \end{align*}
By \eqref{sortres}, Corollary \ref{lexcomp2}, Corollary \ref{lexcomp} and Corollary \ref{lexcomp3}, we have
\begin{align*}
  \YY(\boldsymbol{j})r_i^{+}E^{+}(i)\BUI{\boldsymbol{h}}\in \Theta_{\LENGTH(\boldsymbol{n})-1,\SUM{\boldsymbol{n}}}\VZ 
  + U.
\end{align*}

Let $\boldsymbol{j}=(j_1,\dots,j_p)$. By Corollary \ref{idounan}, 
$\YY(\boldsymbol{j})E^{-}(-i)r_i^{-}\BUI{\boldsymbol{h}}$ belongs to
\begin{align*}
\sum_{i'=1}^{i}
E^{-}(-i')\Theta_{\LENGTH(\boldsymbol{n}),\SUM{\boldsymbol{n}}+i'}\VZ
+\sum_{(i_1,\dots,i_p)\in\mathbb{Z}^p_{\geq 0}, i_1+\dots+i_p=i}
\mathbb{C}Y((j_1-i_1,\dots,j_p-i_p))r_i^{-}\BUI{\boldsymbol{h}}.
\end{align*}
By Remark \ref{sortresRM}, Corollary \ref{lexcomp4} and Corollary \ref{lexcomp3}, we see that
\begin{align*}
Y((j_1-i_1,\dots,j_p-i_p))r_i^{-}\BUI{\boldsymbol{h}}\in
\Theta_{\LENGTH(\boldsymbol{n})-1,\SUM{\boldsymbol{n}}}\VZ+ U.
\end{align*}

By $\YY(\boldsymbol{j})R\BUI{\boldsymbol{h}}
=
c_{\boldsymbol{g}}\BUI{\boldsymbol{n}}
+(\YY(\boldsymbol{j})r\BUI{\boldsymbol{h}}-c_{\boldsymbol{g}}\BUI{\boldsymbol{n}})
+\sum_{i\geq 1}\YY(\boldsymbol{j})r_i^{+}E^{+}(i)\BUI{\boldsymbol{h}}
+\sum_{i\geq 1}\YY(\boldsymbol{j})E^{-}(-i)r_i^{-}\BUI{\boldsymbol{h}}$ and Lemma \ref{sortres3}, 
we have $\BUI{\boldsymbol{n}}\in W_{>\boldsymbol{n}}(\VZ)$.
\end{proof}

\subsection{Forbiddenness of $(-k,-k)$, $(\cint{-k},\cint{-k})$, $(-k-1,-k)$ and $(\cint{-k-1},\cint{-k})$}\label{forone}
By Proposition \ref{mpprop} (\ref{mpprop2}), we have (on any $\GE$-module in a category $\mathcal{C}_K$ for any $K\in\mathbb{C}$)
\begin{align*}
  R_+(-2k)
  = \rho\sum_{i,j\geq 0} E^-(-\alpha_1,-i)\YM{-2k+i-j}E^+(-\alpha_1,j),
\end{align*}
where $\rho=2\Lambda_0(\pi_{(0)}(x_{\alpha_1}))^{2}P_{\alpha_1,\alpha_1}(1)/\Lambda_0(\pi_{(0)}(x_{-\alpha_1}))$, which is a complex number. 
Thus, we have $R_+(-2k)\equiv_{1}0$.
By Example \ref{firstexpansionex} and Proposition \ref{forbhantei} (in this case, $r^{\pm}_i=0$ for $i\geq 1$), 
$(-k,-k)$ is a forbidden pattern.
The other cases are similar.

\subsection{Forbiddenness of $(\cint{-1-3k},-3k)$ and $(-1-3k,\cint{-3k})$}\label{fortwoA}
This follows from
expansions 
below, Proposition \ref{forbhantei} and $R_{1}(-6k-1)=R_{2}(-6k-1)=0$.
{\footnotesize
\begin{align*}
{} &{}  R_1(-6k-1)/(-3(1+2\omega)) \\ 
&\equiv_1 \YM{-1-3k}\YP{-3k} -(1+\omega)\YP{-1-3k}\YM{-3k} + \omega\YM{-2-3k}\YP{1-3k}+\cdots\\
&+ \frac{1}{3}E^-(-1)((2+\omega)\YP{-3k}\YM{-3k} + (-1+\omega)\YM{-1-3k}\YP{1-3k} - (1+2\omega)\YP{-1-3k}\YM{1-3k}+\cdots)\\
&+ \cdots\\
&- \frac{1}{3}((-1+\omega)\YP{-1-3k}\YM{-1-3k} -(1+2\omega) \YM{-2-3k}\YP{-3k}+(2+\omega)\YP{-2-3k}\YM{-3k}+\cdots)E^+(1)\\
&- \cdots,\\
&{} \\
{} &{}  (R_1(-6k-1)+R_2(-6k-1))/(-9) \\ 
&\equiv_1 \YP{-1-3k}\YM{-3k} -\YM{-2-3k}\YP{1-3k} -\YP{-2-3k}\YM{1-3k}+\cdots\\
&+ \frac{1}{3}E^-(-1)(-\YP{-3k}\YM{-3k} -\YM{-1-3k}\YP{1-3k} + 2\YP{-1-3k}\YM{1-3k}+\cdots)\\
&+ \cdots\\
&- \frac{1}{3}(-\YP{-1-3k}\YM{-1-3k} + 2\YM{-2-3k}\YP{-3k} -\YP{-2-3k}\YM{-3k}+\cdots)E^+(1)\\
&- \cdots.
\end{align*}
\normalsize}

Here, $-3(1+2\omega)=C_1(1-\omega^{2(-3k-1)})$, where
$C_1=P_{\alpha_2,-\alpha_1}(1)=P_{\alpha_1,\alpha_1+\alpha_2}(1)(=-3\omega)$ by Remark \ref{koukan} and Example \ref{exppoly2}.
Similarly, we see that $R_2(-6k-1)$ begins with (modulo $\equiv_1$)
$C_2(1-\omega^{1(-3k-1)})\YM{-1-3k}\YP{-3k}$, where $C_2=P_{-(\alpha_1+\alpha_2),-\alpha_1}(1)=P_{\alpha_1,-\alpha_2}(1)(=3(1+\omega))$. Because of
$C_1(1-\omega^{2(-3k-1)})+C_2(1-\omega^{1(-3k-1)})=0$,
$R_1(-6k-1)+R_2(-6k-1)$ begins with
\begin{align*}
(C_1(\omega^{-3k-1}-1)+C_2(\omega^{2(-3k-1)}-1))\YP{-1-3k}\YM{-3k}=-9\YP{-1-3k}\YM{-3k}.
\end{align*}
Throughout this section, we do not explain such routine calculations.

\subsection{Forbiddenness of $(-1-3k,\cint{-1-3k})$ and $(\cint{-2-3k},-3k)$}\label{fortwoB}
This follows from
expansions of $R_{1}(-6k-2)$ and $R_{2}(-6k-2)$
below and Proposition \ref{forbhantei}.
{\footnotesize
\begin{align*}
{} &{}  R_1(-6k-2)/(-3(2+\omega)) \\ 
&\equiv_1 \YP{-1-3k}\YM{-1-3k} + \omega \YM{-2-3k}\YP{-3k} -(1+\omega)\YP{-2-3k}\YM{-3k} +\cdots\\
&+ \frac{1}{3}E^-(-1)((1+2\omega)\YM{-1-3k}\YP{-3k} + (1-\omega)\YP{-1-3k}\YM{-3k} -(2+\omega)\YM{-2-3k}\YP{1-3k}+\cdots)\\
&+ \cdots\\
&- \frac{1}{3}((1-\omega)\YM{-2-3k}\YP{-1-3k} -(2+\omega)\YP{-2-3k}\YM{-1-3k} + (1+2\omega)\YM{-3-3k}\YP{-3k}+\cdots)E^+(1)\\
&- \cdots,\\
&{} \\
{} &{}  (R_1(-6k-2)-(1+\omega)R_2(-6k-2))/(-9\omega) \\ 
&\equiv_1 \YM{-2-3k}\YP{-3k} -\YP{-2-3k}\YM{-3k} -\YM{-3-3k}\YP{1-3k}+\cdots \\
&+ \frac{1}{3}E^-(-1)(2\YM{-1-3k}\YP{-3k} -\YP{-1-3k}\YM{-3k} -\YM{-2-3k}\YP{1-3k}+\cdots)\\
&+ \cdots\\
&- \frac{1}{3}(-\YM{-2-3k}\YP{-1-3k} -\YP{-2-3k}\YM{-1-3k} +2\YM{-3-3k}\YP{-3k}+\cdots)E^+(1)\\
&- \cdots.
\end{align*}
\normalsize}

\subsection{Forbiddenness of $(-2-3k,\cint{-2-3k})$ and $(\cint{-3-3k},-1-3k)$}\label{fortwoC}
This follows from
expansions of $R_{1}(-6k-4)$ and $R_{2}(-6k-4)$
below and Proposition \ref{forbhantei}.
{\footnotesize
\begin{align*}
{} &{}  R_1(-6k-4)/(3(2+\omega)) \\ 
&\equiv_1 \YP{-2-3k}\YM{-2-3k} + \omega \YM{-3-3k}\YP{-1-3k} -(1+\omega) \YP{-3-3k}\YM{-1-3k}+\cdots\\
&+ \frac{1}{3}E^-(-1)((-1+\omega)\YM{-2-3k}\YP{-1-3k} + (2+\omega)\YP{-2-3k}\YM{-1-3k} -(1+2\omega)\YM{-3-3k}\YP{-3k}+\cdots)\\
&+ \cdots\\
&- \frac{1}{3}(-(1+2\omega)\YM{-3-3k}\YP{-2-3k} + (-1+\omega)\YP{-3-3k}\YM{-2-3k} + (2+\omega)\YM{-4-3k}\YP{-1-3k}+\cdots)E^+(1)\\
&- \cdots,\\
&{} \\
{} &{}  (R_1(-6k-4)-(1+\omega)R_2(-6k-4))/(9\omega) \\ 
&\equiv_1 \YM{-3-3k}\YP{-1-3k} -\YP{-3-3k}\YM{-1-3k} -\YM{-4-3k}\YP{-3k}+\cdots\\
&+ \frac{1}{3}E^-(-1)(\YM{-2-3k}\YP{-1-3k} + \YP{-2-3k}\YM{-1-3k} -2\YM{-3-3k}\YP{-3k}+\cdots)\\
&+ \cdots\\
&- \frac{1}{3}(-2\YM{-3-3k}\YP{-2-3k} + \YP{-3-3k}\YM{-2-3k} + \YM{-4-3k}\YP{-1-3k}+\cdots)E^+(1)\\
&- \cdots.
\end{align*}
\normalsize}

\subsection{Forbiddenness of $(\cint{-3-3k},-2-3k)$ and $(-3-3k,\cint{-2-3k})$}\label{fortwoD}
This follows from
expansions of $R_{1}(-6k-5)$ and $R_{2}(-6k-5)$
below and Proposition \ref{forbhantei}.
{\footnotesize
\begin{align*}
{} &{}  R_1(-6k-5)/(3(1+2\omega)) \\ 
&\equiv_1 \YM{-3-3k}\YP{-2-3k} -(1+\omega)\YP{-3-3k}\YM{-2-3k} + \omega \YM{-4-3k}\YP{-1-3k}+\cdots\\
&+ \frac{1}{3}E^-(-1)((1-\omega)\YP{-2-3k}\YM{-2-3k} + (1+2\omega)\YM{-3-3k}\YP{-1-3k} -(2+\omega)\YP{-3-3k}\YM{-1-3k}+\cdots)\\
&+ \cdots\\
&- \frac{1}{3}(-(2+\omega)\YP{-3-3k}\YM{-3-3k} + (1-\omega)\YM{-4-3k}\YP{-2-3k} + (1+2\omega)\YP{-4-3k}\YM{-2-3k}+\cdots)E^+(1)\\
&- \cdots,
\end{align*}
\begin{align*}
{} &{}  (R_1(-6k-5)+R_2(-6k-5))/9 \\ 
&\equiv_1 \YP{-3-3k}\YM{-2-3k} -\YM{-4-3k}\YP{-1-3k} -\YP{-4-3k}\YM{-1-3k}+\cdots\\
&+ \frac{1}{3}E^-(-1)(\YP{-2-3k}\YM{-2-3k} -2\YM{-3-3k}\YP{-1-3k} + \YP{-3-3k}\YM{-1-3k}+\cdots)\\
&+ \cdots\\
&- \frac{1}{3}(\YP{-3-3k}\YM{-3-3k} + \YM{-4-3k}\YP{-2-3k} -2\YP{-4-3k}\YM{-2-3k}+\cdots)E^+(1)\\
&- \cdots.
\end{align*}
\normalsize}

\subsection{Forbiddenness of $(-3-3k,\cint{-3-3k},\cint{-1-3k})$ and $(-5-3k,-3-3k,\cint{-3-3k})$}\label{forthree}
Using Corollary \ref{idounan},
we see that 
\begin{align*}
Z_1(k)=\YP{-3-3k}R_{-}(-4-6k)/3-(R_1(-6k-5)+R_2(-6k-5))\YM{-2-3k}/9
\end{align*}
is expanded as follows.
{\tiny
\begin{align*}
{} &{} Z_1(k)\\
&\equiv_2 \YP{-3-3k}\YM{-3-3k}\YM{-1-3k} + \YM{-4-3k}\YM{-2-3k}\YP{-1-3k} -\YM{-4-3k}\YP{-2-3k}\YM{-1-3k}+\cdots\\
&+ \frac{1}{3}E^-(-1)(-\YP{-2-3k}\YM{-2-3k}\YM{-2-3k} + 2\YM{-3-3k}\YM{-2-3k}\YP{-1-3k} -\YP{-3-3k}\YM{-2-3k}\YM{-1-3k}+\cdots)\\
&+ \cdots\\
&- \frac{1}{3}(-\YP{-3-3k}\YM{-3-3k}\YM{-2-3k} -\YM{-4-3k}\YP{-2-3k}\YM{-2-3k} + 2\YP{-4-3k}\YM{-2-3k}\YM{-2-3k}+\cdots)E^+(1)\\
&- \cdots.
\end{align*}
\normalsize}

The element $(-3-3k,\cint{-3-3k},\cint{-1-3k})$ is a forbidden pattern
because of $Z_1(k)\equiv_2 0$ and that any $\YY$-monomial $\boldsymbol{m}$ appearing in the coefficient of $E^+(i)$ satisfies 
$(-3-3k,\cint{-3-3k},\cint{-1-3k})\LEX\boldsymbol{m}$ (see Proposition \ref{forbhantei}). 

Similarly, using Corollary \ref{idounan}, we see that 
\begin{align*}
Z_2(k)=R_{+}(-8-6k)\YM{-3-3k}/3+\YP{-4-3k}(R_1(-7-6k)+R_2(-7-6k))/9
\end{align*}
is expanded as follows.
{\tiny
\begin{align*}
{} &{} Z_2(k)\\
&\equiv_2 \YP{-5-3k}\YP{-3-3k}\YM{-3-3k} + \YM{-5-3k}\YP{-4-3k}\YP{-2-3k} -\YP{-5-3k}\YM{-4-3k}\YP{-2-3k}+\cdots\\
&+ \frac{1}{3}E^-(-1)(\YP{-4-3k}\YP{-3-3k}\YM{-3-3k} + \YP{-4-3k}\YM{-4-3k}\YP{-2-3k} -2\YP{-4-3k}\YP{-4-3k}\YM{-2-3k}+\cdots)\\
&+ \cdots\\
&- \frac{1}{3}(\YP{-4-3k}\YP{-4-3k}\YM{-4-3k} -2\YM{-5-3k}\YP{-4-3k}\YP{-3-3k} + \YP{-5-3k}\YP{-4-3k}\YM{-3-3k}+\cdots)E^+(1)\\
&- \frac{1}{3}(-2\YM{-5-3k}\YP{-4-3k}\YP{-4-3k} + \YP{-5-3k}\YP{-4-3k}\YM{-4-3k} + \YM{-6-3k}\YP{-4-3k}\YP{-3-3k}+\cdots)E^+(2)\\
&- \cdots.
\end{align*}
\normalsize}

By applying 
\begin{align*}
0\equiv_1 R_+(-8-6k)/3\equiv_1
\YP{-4-3k}\YP{-4-3k}+2\YP{-5-3k}\YP{-3-3k}+\cdots,
\end{align*}
to the coefficient of $E^+(1)$ in $Z_2(k)$, we see that $Z_2(k)\equiv_2 Z'_2(k)$, where
{\tiny
\begin{align*}
{} &{} Z'_2(k)\\
&\equiv_2 \YP{-5-3k}\YP{-3-3k}\YM{-3-3k} + \YM{-5-3k}\YP{-4-3k}\YP{-2-3k} -\YP{-5-3k}\YM{-4-3k}\YP{-2-3k}+\cdots\\
&+ \frac{1}{3}E^-(-1)(\YP{-4-3k}\YP{-3-3k}\YM{-3-3k} + \YP{-4-3k}\YM{-4-3k}\YP{-2-3k} -2\YP{-4-3k}\YP{-4-3k}\YM{-2-3k}+\cdots)\\
&+ \cdots\\
&- \frac{1}{3}(-2\YM{-5-3k}\YP{-4-3k}\YP{-3-3k}-2\YP{-5-3k}\YM{-4-3k}\YP{-3-3k} + \YP{-5-3k}\YP{-4-3k}\YM{-3-3k}+\cdots)E^+(1)\\
&- \frac{1}{3}(-2\YM{-5-3k}\YP{-4-3k}\YP{-4-3k} + \YP{-5-3k}\YP{-4-3k}\YM{-4-3k} + \YM{-6-3k}\YP{-4-3k}\YP{-3-3k}+\cdots)E^+(2)\\
&- \cdots.
\end{align*}
\normalsize}

By applying the same relation to the coefficient of $E^+(2)$ and
by applying 
\begin{align*}
0\equiv_1 R_+(-7-6k)/6\equiv_1
\YP{-4-3k}\YP{-3-3k}+\YP{-5-3k}\YP{-2-3k}+\cdots,
\end{align*}
to the coefficient of $E^+(1)$ in $Z'_2(k)$ in preparation for \S\ref{forfour}, we see that $Z'_2(k)\equiv_2 Z''_2(k)$, where
{\tiny
\begin{align*}
{} &{} Z''_2(k)\\
&\equiv_2 \YP{-5-3k}\YP{-3-3k}\YM{-3-3k} + \YM{-5-3k}\YP{-4-3k}\YP{-2-3k} -\YP{-5-3k}\YM{-4-3k}\YP{-2-3k}+\cdots\\
&+ \frac{1}{3}E^-(-1)(\YP{-4-3k}\YP{-3-3k}\YM{-3-3k} + \YP{-4-3k}\YM{-4-3k}\YP{-2-3k} -2\YP{-4-3k}\YP{-4-3k}\YM{-2-3k}+\cdots)\\
&+ \cdots\\
&- \frac{1}{3}(-2\YP{-5-3k}\YM{-4-3k}\YP{-3-3k}+\YP{-5-3k}\YP{-4-3k}\YM{-3-3k} + 2\YP{-5-3k}\YM{-5-3k}\YP{-2-3k}+\cdots)E^+(1)\\
&- \frac{1}{3}(\YP{-5-3k}\YP{-4-3k}\YM{-4-3k} + 4\YP{-5-3k}\YM{-5-3k}\YP{-3-3k} + \YM{-6-3k}\YP{-4-3k}\YP{-3-3k}+\cdots)E^+(2)\\
&- \cdots.
\end{align*}
\normalsize}

The element $(-5-3k,-3-3k,\cint{-3-3k})$ is a forbidden pattern
because of $Z'_2(k)\equiv_2 Z_2(k)\equiv_2 0$ and that any $\YY$-monomial $\boldsymbol{m}$ appearing in the coefficient of $E^+(i)$ in $Z'_2(k)$ satisfies 
$(-5-3k,-3-3k,\cint{-3-3k})\LEX\boldsymbol{m}$ (see Proposition \ref{forbhantei}). 

\subsection{Forbiddenness of $(\cint{-5-3k},-4-3k,-2-3k,\cint{-1-3k})$}\label{forfour}
By applying 
\begin{align*}
0\equiv_1 R_-(-5-6k)/6
\equiv_1 \YM{-3-3k}\YM{-2-3k}+\YM{-4-3k}\YM{-1-3k}+\cdots,
\end{align*}
to the coefficient of $E^+(1)$ in $Z_1(k)$, we see that $Z_1(k)\equiv_2 Z'_1(k)$, where
{\tiny
\begin{align*}
{} &{} Z'_1(k)\\
&\equiv_2 \YP{-3-3k}\YM{-3-3k}\YM{-1-3k} + \YM{-4-3k}\YM{-2-3k}\YP{-1-3k} -\YM{-4-3k}\YP{-2-3k}\YM{-1-3k}+\cdots\\
&+ \frac{1}{3}E^-(-1)(-\YP{-2-3k}\YM{-2-3k}\YM{-2-3k} + 2\YM{-3-3k}\YM{-2-3k}\YP{-1-3k} -\YP{-3-3k}\YM{-2-3k}\YM{-1-3k}+\cdots)\\
&+ \cdots\\
&- \frac{1}{3}(-\YM{-4-3k}\YP{-2-3k}\YM{-2-3k} +2\YP{-4-3k}\YM{-2-3k}\YM{-2-3k} + 4\YM{-4-3k}\YP{-3-3k}\YM{-1-3k}+\cdots)E^+(1)\\
&- \frac{1}{3}(-\YM{-4-3k}\YP{-3-3k}\YM{-2-3k} -\YP{-4-3k}\YM{-3-3k}\YM{-2-3k} + \YP{-4-3k}\YM{-4-3k}\YM{-1-3k}+\cdots)E^+(2)\\
&- \cdots.
\end{align*}
\normalsize}
	
Again, using Corollary \ref{idounan}, we see that
\begin{align*}
Z_3(k)=Z''_2(k)\YM{-1-3k}-\YP{-5-3k}Z'_1(k) 
\end{align*}
is expanded as follows.
{\tiny
\begin{align*}
{} &{} Z_3(k)\\
&\equiv_3 \YM{-5-3k}\YP{-4-3k}\YP{-2-3k}\YM{-1-3k} -\YP{-5-3k}\YM{-4-3k}\YM{-2-3k}\YP{-1-3k} 
+\cdots\\
&+ \frac{1}{3}E^-(-1)(\YP{-4-3k}\YP{-3-3k}\YM{-3-3k}\YM{-1-3k} + \YP{-4-3k}\YM{-4-3k}\YP{-2-3k}\YM{-1-3k} 
+\cdots)\\
&+ \cdots\\
&- \frac{1}{3}(\YP{-5-3k}\YM{-4-3k}\YP{-2-3k}\YM{-2-3k} -2\YP{-5-3k}\YP{-4-3k}\YM{-2-3k}\YM{-2-3k} 
+\cdots)E^+(1)\\
&- \frac{1}{3}(\YP{-5-3k}\YM{-4-3k}\YP{-3-3k}\YM{-2-3k} + \YP{-5-3k}\YP{-4-3k}\YM{-3-3k}\YM{-2-3k} 
+\cdots)E^+(2)\\
&- \cdots.
\end{align*}
\normalsize}

The element $(\cint{-5-3k},-4-3k,-2-3k,\cint{-1-3k})$ is a forbidden pattern
because of $Z_3(k)\equiv_3 0$ and that any $\YY$-monomial $\boldsymbol{m}$ appearing in the coefficient of $E^+(i)$ satisfies 
$(\cint{-5-3k},-4-3k,-2-3k,\cint{-1-3k})\LEX\boldsymbol{m}$ (see Proposition \ref{forbhantei}).

\section{An automatic derivation of $q$-difference equations via ~\cite{TT}}\label{auto}
Let $\Sigma$ be a nonempty finite set, called an alphabet in the context of formal language theory.
We denote by $\Sigma^\ast$ the set of finite words $\bigsqcup_{n\geq 0}\Sigma^n$
of $\Sigma$. A word $(w_1,\dots,w_n)\in\Sigma^n$ of length $n$ is written as $w_1\cdots w_n$.
By the word concatenation and the empty word $\EMPTYWORD$,
we regard the set $\Sigma^{\ast}$ as a free monoid generated by $\Sigma$.
For 
$X,Y\subseteq\Sigma^{\ast}$, we put
$XY=\{wv\mid w\in X,v\in Y\}$.

\begin{Def}
A deterministic finite automaton (DFA, for short) over $\Sigma$
is a 5-tuple $(Q,\Sigma,\delta,s,F)$, where $Q$ is a finite set (called the set of states),
$\delta:Q\times \Sigma\to Q$ is a function (called the transition function), $s$ is an element of $Q$ (called
the start state) and $F$ is a subset of $Q$ (called the set of accept states).
\end{Def}

For a DFA $M=(Q,\Sigma,\delta,s,F)$, the language recognized by $M$ is defined as
\begin{align*}
L(M)=\{w\in\Sigma^{\ast}\mid \EXTE{\delta}(s,w)\in F\},
\end{align*}
where $\EXTE{\delta}(q,w)$ is defined inductively by $\EXTE{\delta}(q,\EMPTYWORD)=q$ and 
$\EXTE{\delta}(q,w_1v)=\EXTE{\delta}(\delta(q,w_1),v)$ for $q\in Q, w_1\in\Sigma$ and $w,v\in\Sigma^{\ast}$. 






We denote the set of 2-colored partitions (see \S\ref{mainse}) by $\TPAR$.
Let $I=\{\mya,\myb,\dots,\mym\}$ and
consider the map $\pi:I\to \TPAR$ defined by (see also Example \ref{exbipar})
\begin{align*}
\mya\mapsto\emptypar,\quad
\myb\mapsto(1),\quad
\myc\mapsto(\cint{1}),\quad
\myd\mapsto(2),\quad
\mye\mapsto(\cint{2}),\quad
\myf\mapsto(3),\quad
\myg\mapsto(\cint{3}),\\
\myh\mapsto(2,\cint{1}),\quad
\myi\mapsto(\cint{2},1),\quad
\myj\mapsto(3,1),\quad
\myk\mapsto(3,\cint{1}),\quad
\myl\mapsto(\cint{3},\cint{1}),\quad
\mym\mapsto(3,\cint{3}).
\end{align*}

Let $\SEQ(I)=I^{\mathbb{Z}_{\geq 1}}(=\{(i_1,i_2,\dots)\mid \textrm{$i_j\in I$ for $j\geq 1$}\})$. We denote by
$\SEQ(I,\pi)$ the subset of $\SEQ(I)$,
which consists of $\boldsymbol{i}=(i_1,i_2,\dots)$ 
such that
$\{k\geq 1\mid i_k\ne\mya\}$ is a finite set. For a nonnegative integer $t$, we define
$\SHIFT_t(\lambda)=(\lambda_1+t,\dots,\lambda_{\ell}+t)$ (see Corollary \ref{lexcomp4})
for a 2-colored partition $\lambda=(\lambda_1,\dots,\lambda_{\ell})$.
Then, the map 
\begin{align*}
\PAI:\SEQ(I,\pi)\to \TPAR, 
\end{align*}
which sends $\boldsymbol{i}=(i_1,i_2,\dots)\in \SEQ(I,\pi)$ to the smallest (with respect to the length)
2-colored partition that contains $\SHIFT_{3k-3}(\pi(i_k))$ for any $k\geq 1$ such that $i_k\ne\mya$
is well-defined (see ~\cite[Definition 2.5.(2)]{TT}). 
It is clearly an injection. 

\begin{Ex}
$\PAI(\boldsymbol{i})=(\cint{11},10,\cint{8},3,\cint{3})$ for $\boldsymbol{i}=(\mym,\mya,\mye,\myi,\mya,\mya,\dots)$.
\end{Ex}

The translation result below is verified elementarily,
and we omit a proof.

\begin{Prop}\label{iikae}
For $\boldsymbol{i}\in\SEQ(I,\pi)$, the condition $\PAI(\boldsymbol{i})\in\BIR$ is equivalent to
the condition that $\boldsymbol{i}$ does not contain any of the word in $J$, where
\begin{align*}
J &= \{\myf\myb,\myg\myb,\myj\myb,\myk\myb,\myl\myb,\mym\myb,\myl\myc,\mym\myc,\myj\myc,\myk\myc,\myf\myc,\myg\myc,
\mym\myd,\myf\mye,\myj\mye,\myk\mye,\mym\mye,\myh\myi,\myf\myi,\myg\myi,\myj\myi,\myk\myi,\myl\myi,\mym\myi,\\
&\quad\quad \myf\myh,\myg\myh,\myj\myh,\myk\myh,\myl\myh,\mym\myh,\myf\myj,\myg\myj,\myj\myj,\myk\myj,\myl\myj,\mym\myj,
\myf\myk,\myg\myk,\myj\myk,\myk\myk,\myl\myk,\mym\myk,\myf\myl,\myg\myl,\myj\myl,\myk\myl,\myl\myl,\mym\myl
\}.
\end{align*}
Moreover, any $\lambda\in\BIR$ is obtained in this way (i.e., $\lambda=\PAI(\boldsymbol{i})$ for some 
$\boldsymbol{i}\in\SEQ(I,\pi)$).
\end{Prop}

Let $M=(Q,I,\delta,s,F)$ be the minimum DFA (with respect to the number of states)
over $I$ such that $L(M)=I^{\ast}JI^{\ast}$. 
It is computable by the standard algorithms in formal language theory (see ~\cite[Appendix A]{TT} for a review).
By running them (see Example \ref{gap} below), we see automatically that $Q=\{\JOO{0},\dots,\JOO{5}\}$, $s=\JOO{0}$, $F=\{\JOO{1}\}$ and that $\delta(q,j)$ for 
$q\in Q, j\in I$ is given by the table below.
\begin{center}
\begin{tabular}{r|rrrrrrrrrrrrr}
$q\backslash j$ & $\mya$ & $\myb$ & $\myc$ & $\myd$ & $\mye$ & $\myf$ & $\myg$ & $\myh$ & $\myi$ & $\myj$ & $\myk$ & $\myl$ & $\mym$ \\ \hline
\JO{0} & \JO{0} & \JO{0} & \JO{0} & \JO{0} & \JO{0} & \JO{2} & \JO{4} & \JO{5} & \JO{0} & \JO{2} & \JO{2} & \JO{4} & \JO{3} \\ 
\JO{1} & \JO{1} & \JO{1} & \JO{1} & \JO{1} & \JO{1} & \JO{1} & \JO{1} & \JO{1} & \JO{1} & \JO{1} & \JO{1} & \JO{1} & \JO{1} \\ 
\JO{2} & \JO{0} & \JO{1} & \JO{1} & \JO{0} & \JO{1} & \JO{2} & \JO{4} & \JO{1} & \JO{1} & \JO{1} & \JO{1} & \JO{1} & \JO{3} \\ 
\JO{3} & \JO{0} & \JO{1} & \JO{1} & \JO{1} & \JO{1} & \JO{2} & \JO{4} & \JO{1} & \JO{1} & \JO{1} & \JO{1} & \JO{1} & \JO{3} \\ 
\JO{4} & \JO{0} & \JO{1} & \JO{1} & \JO{0} & \JO{0} & \JO{2} & \JO{4} & \JO{1} & \JO{1} & \JO{1} & \JO{1} & \JO{1} & \JO{3} \\ 
\JO{5} & \JO{0} & \JO{0} & \JO{0} & \JO{0} & \JO{0} & \JO{2} & \JO{4} & \JO{5} & \JO{1} & \JO{2} & \JO{2} & \JO{4} & \JO{3} 
\end{tabular}
\end{center}

\begin{Ex}\label{gap}
An execution for obtaining the minimum DFA using
a GAP package Automata~\cite{DHLM} is as follows.

\begin{small}
\noindent\verb|A:=RationalExpression("fbUgbUjbUkbUlbUmbUlcUmcUjcUkcUfcUgcUmdUfeUjeUkeUmeUhiUfiUgiUjiUkiUliU|

\noindent\verb|miUfhUghUjhUkhUlhUmhUfjUgjUjjUkjUljUmjUfkUgkUjkUkkUlkUmkUflUglUjlUklUllUml","abcdefghijklm");|

\noindent\verb|Is:=RationalExpression("(aUbUcUdUeUfUgUhUiUjUkUlUm)*","abcdefghijklm");|

\noindent\verb|r:=ProductRatExp(Is,ProductRatExp(A,Is)); M:=RatExpToAut(r); Display(M);|
\end{small}
\end{Ex}

For each $q\in Q\setminus F(=\{\JOO{0},\JOO{2},\JOO{3},\JOO{4},\JOO{5}\})$, we define
\begin{align*}
\LC{q}=\PAI(\SEQ(I,\pi)\cap (L(M_q)^{c})^{\wedge}),
\end{align*}
where, for $A\subseteq I^{\ast}$, we put (see \cite[\S3.2]{TT})
\begin{align*}
A^{\wedge} = \{\boldsymbol{i}\in\SEQ(I)\mid i_1\cdots i_n\in A\textrm{ for all }n\geq 1\},
\end{align*}
and, for a DFA $M=(I,\Sigma,\delta,s,F)$ and $v\in I$, we put $M_v=(I,\Sigma,\delta,v,F)$, i.e.,
the DFA obtained from $M$ by changing the start state to $v$ (see \cite[Definition 3.11]{TT}).

\begin{Rem}
By Proposition \ref{iikae}, we have $\LC{q_0}=\BIR$.
Moreover, we accidentally have $\LC{q_2}=\BIRP$ thanks to the following two facts.
\begin{enumerate}
\item For $\boldsymbol{i}\in\SEQ(I,\pi)$ such that $\PAI(\boldsymbol{i})\in\BIR$, 
the condition $\PAI(\boldsymbol{i})\in\BIRP$ is equivalent to
the condition $i_1\ne \myb,\myc,\mye,\myh,\myi,\myj,\myk,\myl$, where $\boldsymbol{i}=(i_1,i_2,\dots)$.
\item For $t\in\Sigma$, the condition $\delta(\JOO{2},t)\not\in F$ is equivalent to
the condition $t\ne \myb,\myc,\mye,\myh,\myi,\myj,\myk,\myl$. Moreover, $\delta(\JOO{2},\mya)=\JOO{0}$.
\end{enumerate}
\end{Rem}

By ~\cite[Theorem 3.17]{TT}, we have 
\begin{align*}
f_{\LC{q}}(x,q)=\sum_{a\in I, \delta(q,a)\not\in F} x^{\ell(\pi(a))}q^{|\pi(a)|}f_{\LC{\delta(q,a)}}(xq^3,q).
\end{align*}
Putting $F_i(x)=f_{\LC{\JOO{i}}}(x,q)$, we have the simultaneous $q$-difference equation
\begin{align*}
\begin{pmatrix}
F_0(x) \\
F_2(x) \\
F_3(x) \\
F_4(x) \\
F_5(x)
\end{pmatrix}
=
\begin{pmatrix}
1+2xq+2xq^2+x^2q^3 & xq^3+2x^2q^4 & x^2q^6 & xq^3+x^2q^4 & x^2q^3 \\
1+xq^2             & xq^3        & x^2q^6 & xq^3        & 0      \\
1                  & xq^3        & x^2q^6 & xq^3        & 0      \\
1+2xq^2            & xq^3        & x^2q^6 & xq^3        & 0      \\
1+2xq+2xq^2        & xq^3+2x^2q^4 & x^2q^6 & xq^3+x^2q^4 & x^2q^3 \\
\end{pmatrix}
\begin{pmatrix}
F_0(xq^3) \\
F_2(xq^3) \\
F_3(xq^3) \\
F_4(xq^3) \\
F_5(xq^3)
\end{pmatrix}.
\end{align*}


By the Murray-Miller algorithm (see ~\cite[Appendix B]{TT}), we get
the following.




\begin{Prop}\label{PropqdiffR}
The generating functions $f_{\BIR}(x,q)$ and $f_{\BIRP}(x,q)$ satisfy the following $q$-difference equations respectively.

\tiny
\begin{align*}
{} &{} (2+3xq^{4}+xq^{6})f_{\BIR}(x,q)\\
&=
(2+4xq+4xq^{2}+4xq^{3}+3xq^{4}+xq^{6}+4x^{2}q^{3}
+6x^{2}q^{4}+6x^{2}q^{5}+8x^{2}q^{6}+2x^{2}q^{7}
+2x^{2}q^{8}+2x^{2}q^{9}+6x^{3}q^{7}+2x^{3}q^{9}+3x^{3}q^{10}+x^{3}q^{12})f_{\BIR}(xq^{3},q)\\
&- x^{2}q^{7}(2+2q+3xq+4xq^{2}+xq^{3}+4xq^{4}-xq^{5}+4xq^{6}
+xq^{7}+2x^{2}q^{5}+6x^{2}q^{7}+6x^{2}q^{8}+2x^{2}q^{9}+2x^{2}q^{10}+4x^{2}q^{11}+3x^{3}q^{9}
+x^{3}q^{11}\\
&+6x^{3}q^{12}+2x^{3}q^{14})f_{\BIR}(xq^{6},q)+ x^{4}q^{21}(1-xq^6)^2(2+3xq+xq^{3})f_{\BIR}(xq^{9},q),\\
{} &{} (1+xq^{5}+xq^{7})f_{\BIRP}(x,q) \\
&+(1+xq^{2}+2xq^{3}+2xq^{4}+2xq^{5}+xq^{7}+3x^{2}q^{6}+2x^{2}q^{7}+3x^{2}q^{8}+3x^{2}q^{9}
+2x^{2}q^{10}+x^{2}q^{11}+x^{2}q^{12}+2x^{3}q^{11}+2x^{3}q^{13}+x^{3}q^{14}+x^{3}q^{16})f_{\BIRP}(xq^{3},q)\\
&-x^{2}q^{10}(1+q+xq^{2}+xq^{3}+2xq^{4}+2xq^{6}+xq^{7}+xq^{8}-x^{2}q^{7}+2x^{2}q^{8}+x^{2}q^{9}
+2x^{2}q^{10}+3x^{2}q^{11}+x^{2}q^{12}+x^{2}q^{13}+2x^{2}q^{14}+x^{3}q^{13}+x^{3}q^{15}\\
&+2x^{3}q^{16}+2x^{3}q^{18})f_{\BIRP}(xq^{6},q)+x^{4}q^{27}(1-xq^6)(1-xq^9)(1+xq^{2}+xq^{4})f_{\BIRP}(xq^{9},q)
\end{align*}
\normalsize
\end{Prop}

\section{The cylindric partitions}\label{cylin}

Let $r\geq 1$. Recall that
a profile $\boldsymbol{c}=(c_i)_{i\in\mathbb{Z}/r\mathbb{Z}}$ is a family of nonnegative integers (indexed by $\mathbb{Z}/r\mathbb{Z}$).
Recall also that a cylindric partition~\cite{GK} $\boldsymbol{\lambda}=(\lambda^{(i)})_{i\in\mathbb{Z}/r\mathbb{Z}}$ of 
a profile $\boldsymbol{c}=(c_i)_{i\in\mathbb{Z}/r\mathbb{Z}}$ is 
a family of partitions 
such that
$\lambda^{(i)}_j\geq \lambda^{(i+1)}_{j+c_{i+1}}$
for $i\in\mathbb{Z}/r\mathbb{Z}$ and $j\geq 1$. 
As usual, we put $\mu_k=0$ for a partition $\mu$ if $k>\ell(\mu)$.
We define
\begin{align*}
F_{\boldsymbol{c}}(x,q)=\sum_{\boldsymbol{\lambda}\in\CP{\boldsymbol{c}}}x^{\max(\boldsymbol{\lambda})}q^{\SUM{\boldsymbol{\lambda}}},\quad
G_{\boldsymbol{c}}(x,q)=(xq;q)_{\infty}F_{\boldsymbol{c}}(x,q),
\end{align*}
where $\CP{\boldsymbol{c}}$ be the set of cylindric partitions of profile 
$\boldsymbol{c}$, 
$\max(\boldsymbol{\lambda})=\max\{\lambda^{(i)}_1,\dots,\lambda^{(i)}_{\ell(\lambda^{(i)})}\mid i\in\mathbb{Z}/r\mathbb{Z}\}$ and
$\SUM{\boldsymbol{\lambda}}=\sum_{i\in\mathbb{Z}/r\mathbb{Z}}|\lambda^{(i)}|$
for a cylindric partition $\boldsymbol{\lambda}$.

\begin{Thm}[{\cite[Proposition 5.1]{Bor}}]\label{borp}
For a profile $\boldsymbol{c}=(c_i)_{i\in\mathbb{Z}/r\mathbb{Z}}$, we have
\begin{align*}
F_{\boldsymbol{c}}(1,q)=\frac{1}{(q;q)_{\infty}}\chi_{A^{(1)}_{r-1}}(\sum_{i\in\mathbb{Z}/r\mathbb{Z}}c_i\Lambda_i).
\end{align*}
\end{Thm}

\begin{Thm}[{\cite[Proposition 3.1$+$(3.5)]{CW}}]\label{cwrec}
For a profile $\boldsymbol{c}=(c_i)_{i\in\mathbb{Z}/r\mathbb{Z}}$, we have
\begin{align*}
G_{\boldsymbol{c}}(x,q)=\sum_{\emptyset\ne J\subseteq I_{\boldsymbol{c}}}(-1)^{|J|-1}(xq;q)_{|J|-1}G_{\boldsymbol{c}(J)}(xq^{|J|}),
\end{align*}
where $I_{\boldsymbol{c}}=\{i\in\mathbb{Z}/r\mathbb{Z}\mid c_i>0\}$ and $\boldsymbol{c}(J)=(c_i(J))_{i\in\mathbb{Z}/r\mathbb{Z}}$ 
is defined by
\begin{align*}
c_i(J)=
\begin{cases}
c_i-1 & \textrm{if $i\in J$ and $(i-1)\not\in J$},\\
c_i+1 & \textrm{if $i\not\in J$ and $(i-1)\in J$},\\
c_i & \textrm{otherwise}.
\end{cases}
\end{align*}
\end{Thm}

In this paper, when a profile is written in one-line format $(c_0,\dots,c_{r-1})$, the nonnegative integer $c_i$ is the one indexed by $i+r\mathbb{Z}\in\mathbb{Z}/r\mathbb{Z}$ in the profile. 

\begin{Prop}\label{PropqdiffG}
The generating functions $G_{(3,0,0)}(x,q)$ and $G_{(1,1,1)}(x,q)$ satisfy the following $q$-difference equations respectively.
\begin{align*}
{} &(1+xq^5)G_{(3,0,0)}(x,q)\\
&=(1+xq^{2}+2xq^{3}+2xq^{4}+2xq^{5}+2x^{2}q^{6}+2x^{2}q^{7}+2x^{2}q^{8}+x^{2}q^{9}+x^{3}q^{11})G_{(3,0,0)}(xq^{3},q)\\
&\quad +xq^6(1+xq^2)(1-xq^4)(1-xq^5)G_{(3,0,0)}(xq^{6},q),\\
{} &(1+xq^4)G_{(1,1,1)}(x,q)\\
&= (1+2xq+2xq^{2}+2xq^{3}+xq^{4}+x^{2}q^{3}+2x^{2}q^{4}+2x^{2}q^{5}+2x^{2}q^{6}+x^{3}q^{7})G_{(1,1,1)}(xq^{3},q)\\
&\quad +xq^{3}(1+xq)(1-xq^4)(1-xq^5)G_{(1,1,1)}(xq^6,q).
\end{align*}
\end{Prop}

\begin{proof}\label{cwrecex}
By the Corteel-Welsh recursion formula (Theorem \ref{cwrec}), we have
\begin{align*}
G_{(3,0,0)}(x) &= G_{(2,1,0)}(xq), \\
G_{(2,1,0)}(x) &= 2G_{(2,0,1)}(xq)-(1-xq)G_{(1,1,1)}(xq^2), \\
G_{(2,0,1)}(x) &= G_{(3,0,0)}(xq)+G_{(1,1,1)}(xq)-(1-xq)G_{(2,1,0)}(xq^2), \\
G_{(1,1,1)}(x) &= 3G_{(2,1,0)}(xq)-3(1-xq)G_{(2,0,1)}(xq^2)+(1-xq)(1-xq^2)G_{(1,1,1)}(xq^3).
\end{align*}
Thus, we easily see
\begin{align}
\begin{pmatrix}
G_{(3,0,0)}(x) \\
G_{(1,1,1)}(x) 
\end{pmatrix}
=
\begin{pmatrix}
2xq^3 & 1+xq^2 \\
3xq^3(1+xq) & 1+2xq+2xq^2+x^2q^3
\end{pmatrix}
\begin{pmatrix}
G_{(3,0,0)}(xq^3) \\
G_{(1,1,1)}(xq^3) 
\end{pmatrix}.
\label{Grel}
\end{align}
By the Murray-Miller algorithm (see ~\cite[Appendix B]{TT}), we get
the result.
\end{proof}



\section{Andrews-Gordon type series}\label{ags}

\subsection{Certificate recurrences}\label{certe}
Let $r\geq 2$ and we denote the set of maps from $\mathbb{Z}^r$ to $\mathbb{Q}(q)$
by $\MAP(\mathbb{Z}^r,\mathbb{Q}(q))$, with the variables $n,k_2,\dots,k_r$ in this order.

Let $f\in \MAP(\mathbb{Z}^r,\mathbb{Q}(q))$.
The shift operators $\ENU$, $\KEI{2},\dots,\KEI{r}$ are defined by
\begin{align*}
\ENU f(n,k_2,\dots,k_r) &= f(n-1,k_2,\dots,k_r),\\
\KEI{i} f(n,k_2,\dots,k_r) &= f(n,k_2,\dots,k_{i-1},k_i-1,k_{i+1},\dots,k_r)
\end{align*}
for $2\leq i\leq r$.
As usual, we say that $f$ is summable if 
$\{(k_2,\dots,k_r)\in\mathbb{Z}^{r-1}\mid f(n,k_2,\dots,k_r)\ne 0\}$
is a finite set for any $n\in\mathbb{Z}$. In this case, 
\begin{align}
f_n=\sum_{k_2,\dots,k_r\in\mathbb{Z}}f(n,k_2,\dots,k_r)
\label{efuenu}
\end{align}
is well-defined.
Let $\NONC$ be a $\mathbb{Q}(q)$-subalgebra in $\END_{\textrm{$\mathbb{Q}(q)$-lin}}(\MAP(\mathbb{Z}^r,\mathbb{Q}(q)))$
generated by the shift operators $\ENU$, $\KEI{2},\dots,\KEI{r}$ and the multiplications $q^n$, $q^{k_2},\dots,q^{k_r}$.
Note that in $\NONC$ we have
\begin{align}
\begin{split}
  \ENU\KEI{i}=\KEI{i}\ENU,\quad
q^{k_i}\ENU=\ENU q^{k_i},\quad
q^{n}\KEI{i}=\KEI{i} q^{n},\\
\ENU q^{n} = q^{n-1}\ENU,\quad
\KEI{j} q^{k_i} = q^{k_i-\delta_{i,j}}\KEI{j},\quad
q^{n}q^{k_i}=q^{k_i}q^{n},
\end{split}
\label{noncrel}
\end{align}
for $2\leq i,j\leq r$.
We denote by $\mathbb{Q}(q)[q^n]$ the $\mathbb{Q}(q)$-subalgebra generated by $q^n$ in $\NONC$, which 
is clearly a polynomial $\mathbb{Q}(q)$-algebra generated by $q^n$.
The following is standard in
deriving a recurrence relation (see ~\cite[\S3]{Rie}). 
For completeness, we duplicate a proof.
\begin{Prop}\label{certprop}
Assume $f\in \MAP(\mathbb{Z}^r,\mathbb{Q}(q))$ is summable and 
there exists $J\geq 0$, $p_0,\dots,p_J\in\mathbb{Q}(q)[q^n]$ and $C_2,\dots,C_r\in\NONC$ such that
\begin{align}
\Big(\sum_{j=0}^{J}p_j\ENU^j + \sum_{i=2}^{r}(1-\KEI{i})C_i\Big)f=0.
\label{certr}
\end{align}
Then, we have
a $q$-holonomic recurrence (see \eqref{efuenu})
$f_n+p_1f_{n-1}+\dots+p_Jf_{n-J}=0$.
\end{Prop}

\begin{proof}
Let $G_i=C_if$ for $2\leq i\leq r$ and note that it is summable.
For $n\in\mathbb{Z}$, take $M>0$ so that 
$(k_2,\dots,k_r)\not\in\{-M,-M+1,\dots,M-1,M\}^{r-1}$ implies
\begin{align*}
0=f(n,k_2,\dots,k_r)=\cdots=f(n-J,k_2,\dots,k_r)=G_2(n,k_2,\dots,k_r)=\dots=G_r(n,k_2,\dots,k_r).
\end{align*}
By summation $\sum_{k_2=-M}^{M+1}\dots\sum_{k_r=-M}^{M+1}$ of 
$\sum_{j=0}^{J}p_jf(n-j,k_2,\dots,k_r)$, which is equal to 
$-\sum_{i=2}^{r}(G_i(n,k_2,\dots,k_i,\dots,k_r)-G_i(n,k_2,\dots,k_i-1,\dots,k_r))$,
we get the result.
\end{proof}


\subsection{Automatic derivation of $q$-difference equations via $q$-Sister Celine}\label{sister}
The goal of this subsection is to give a brief review of
an automatic calculation of a $q$-difference equation of
an Andrews-Gordon type series
\begin{align}
  H(x,q)
  =\sum_{n\in\mathbb{Z}}h_n(q)x^n
  =\sum_{i_1,\dots,i_{\ell}\geq 0}
  \frac{q^{h(i_1,\dots,i_{\ell})+f_1i_1+\dots+f_{\ell}i_{\ell}}}{(q^{d_1};q^{d_1})_{i_1}\cdots(q^{d_{\ell}};q^{d_{\ell}})_{i_{\ell}}}x^{c_1i_1+\dots+c_{\ell}i_{\ell}},
\label{agh}
\end{align}
where $\ell\geq 2$, $f_1,\dots,f_\ell$ are 
integers, $c_1,\dots,c_\ell,d_1,\dots,d_{\ell}$ are positive integers such that $c_1=1$, and $h(i_1,\dots,i_{\ell})$ is a 
quadratic term, which takes an integer value for any $(i_1,\dots,i_{\ell})\in\mathbb{Z}^{\ell}$.

The algorithm is a typical application of $q$-version of the Sister Celine technique~\cite{Rie}
to a $q$-proper hypergeometric term~\cite[Definition 2.3]{Rie}
\begin{align*}
  F(n,i_2,\dots,i_{\ell})=
  \frac{q^{G(n,i_2,\dots,i_{\ell})+z_1n+z_2i_2+\dots+z_{\ell}i_{\ell}}}{(q^{d_1};q^{d_1})_{n-(c_2i_2+\dots+c_{\ell}i_{\ell})}(q^{d_2};q^{d_2})_{i_2}\cdots(q^{d_{\ell}};q^{d_{\ell}})_{i_{\ell}}},
\end{align*}
where
the quadratic term $G$ and $z_1,\dots,z_{\ell}\in\mathbb{Z}$ are determined by
the equation
\begin{align*}
  G(n,i_2,\dots,i_{\ell})+z_1n+ z_2i_2+\dots+z_{\ell}i_{\ell}
  =h(i_1,i_2,\dots,i_{\ell})+f_1i_1+f_2i_2+\dots+f_{\ell}i_{\ell},
\end{align*}
and $i_1=n-(c_2i_2+\dots+c_{\ell}i_{\ell})$. 
Note that $F(n,i_2,\dots,i_{\ell})$ is well-defined for any arguments
by the convention $\frac{1}{(q^d;q^d)_{-m}}=0$ for positive integers $d$ and $m$ (See ~\cite[p.~3]{Yen} and ~\cite[\S2.1]{Rie}).
In the following, we expand the quadratic term $G$ as
\begin{align*}
  G(a,b_2\dots,b_{\ell})
=\sum_{t=1}^{\ell}g_{t,t}{b_t\choose 2}
  +\sum_{1\leq t<u\leq \ell}g_{t,u}b_tb_u,
\end{align*}
where we put $b_1=a$.
We also put $g_{u,t}=g_{t,u}$ for $1\leq t<u\leq\ell$ as a convenience.

\begin{Def}[{\cite[Definition 2.5]{Rie}}]
For a finite subset $S\subseteq \mathbb{Z}_{\geq 0}^{\ell}$ such that
$(0,\dots,0)\in S$ and for 
$(i,j_2,\dots,j_{\ell})\in S$, let
a Laurent polynomial $T^{(S)}_{i,j_2,\dots,j_{\ell}}\in\mathbb{Q}(q)[Q_1,Q_1^{-1},\dots,Q_{\ell},Q_{\ell}^{-1}]$ be
\begin{align*}
T^{(S)}_{i,j_2,\dots,j_{\ell}} &= q^{-z_1i-z_2j_2-z_{\ell}j_{\ell}}
(q^{-M_Sd_1}Q_1^{d_1}Q_2^{-c_2d_1}\cdots Q_{\ell}^{-c_{\ell}d_1};q^{-d_1})_{i-c_2j_2-\dots -c_{\ell}j_{\ell}-M_S}\\
&\quad
\Big(\prod_{t=2}^{\ell}(Q_t^{d_t};q^{-d_t})_{j_{t}}\Big)
\Big(q^{\sum_{t=1}^{\ell}g_{t,t}{j_t+1\choose 2}+\sum_{1\leq t<u\leq \ell}g_{t,u}j_tj_{u}}\Big)
\Big(\prod_{t=1}^{\ell}Q_t^{-\sum_{u=1}^{\ell} g_{t,u}j_{u}}\Big),
\end{align*}
where $j_1=i$ and $M_S=\min\{i-c_2j_2-\dots -c_{\ell}j_{\ell}\mid (i,j_2,\dots,j_{\ell})\in S\}$.
\end{Def}


\begin{Ex}\label{exsister}
For an Andrews-Gordon type series 
\begin{align*}
H(x,q)=
\sum_{s,t\geq 0}
\frac{q^{s^2+2t^2+2st}}{(q^2;q^2)_s(q^2;q^2)_t}x^{s+t},
\end{align*}
we take a $q$-proper hypergeometric term
\begin{align*}
F(n,i_2)
= \frac{q^{(n-i_2)^2+2i_2^2+2(n-i_2)i_2}}{(q^2;q^2)_{n-i_2}(q^2;q^2)_{i_2}}
= \frac{q^{2{n\choose 2}+2{i_2\choose 2}}}{(q^2;q^2)_{n-i_2}(q^2;q^2)_{i_2}}q^{n}q^{i_2}, 
\end{align*}
i.e., $z_1=z_2=1$, $g_{1,1}=g_{2,2}=2$, $g_{1,2}=0(=g_{2,1})$, and $c_1=c_2=1$, $d_1=d_2=2$.

For $S=\{(0,0),(0,1),(1,0),(1,1)\}$, we have $M_S=-1$ and
\begin{align*}
  T^{(S)}_{0,0} = 1-q^2Q_1^2Q_2^{-2},\quad
  T^{(S)}_{1,0} = q^{-1}(1-q^2Q_1^2Q_2^{-2})(1-Q_1^2Q_2^{-2})q^{2}Q_1^{-2},\\
T^{(S)}_{0,1} = q^{-1}(1-Q_2^2)q^2Q_2^{-2},\quad
T^{(S)}_{1,1} = q^{-2}(1-q^2Q_1^2Q_2^{-2})(1-Q_2^2)q^{4}Q_1^{-2}Q_2^{-2}.
\end{align*}
\end{Ex}

\begin{Prop}\label{kurikaeshi}
Assume that, in $\mathbb{Q}(q)[Q_1^{\pm1},\dots,Q_{\ell}^{\pm1}]$, we have
\begin{align}
\sum_{(i,j_2,\dots,j_{\ell})\in S}
\sigma_{i,j_2,\dots,j_{\ell}}T^{(S)}_{i,j_2,\dots,j_{\ell}}
=0
\label{algoeqmain}
\end{align}
for an $S$-tuple of Laurent polynomials $(\sigma_{i,j_2,\dots,j_{\ell}})_{(i,j_2,\dots,j_{\ell})\in S}\in(\mathbb{Q}(q)[Q_1^{\pm1},Q_2^{\pm1},\dots,Q_{\ell}^{\pm1}])^{S}$.
Then, in $\mathbb{Q}(q)$, we have, for $n,i_2,\dots,i_{\ell}\in\mathbb{Z}$,
\begin{align*}
\sum_{(i,j_2,\dots,j_{\ell})\in S}\sigma_{i,j_2,\dots,j_{\ell}}(q^n,q^{i_2},\dots,q^{i_{\ell}})
F(n-i,i_2-j_2,\dots,i_{\ell}-j_{\ell})=0.
\end{align*}
\end{Prop}

\begin{proof}
Note that the left hand side of \eqref{algoeqmain} is
a slightly expanded version of 
$L_{F,S}$ in ~\cite{Rie} (see ~\cite[Definition 2.5]{Rie}),
based on the assumptions of $S$. The claim is exactly a combination of ~\cite[Theorem 2.1]{Rie} and ~\cite[Lemma 2.1]{Rie} (and the remark following ~\cite[Lemma 2.1]{Rie}).
\end{proof}

For a Laurent polynomial $\tau\in\mathbb{Q}(q)[Q_1,Q_1^{-1},\dots,Q_{\ell},Q_{\ell}^{-1}]$,
we expand as
\begin{align*}
\tau=
\sum_{(p_2,\dots,p_{\ell})\in\mathbb{Z}^{\ell-1}}
\tau^{(p_2,\dots,p_{\ell})}Q_2^{p_2}\dots Q_{\ell}^{p_{\ell}},
\end{align*}
where $\tau^{(p_2,\dots,p_{\ell})}\in\mathbb{Q}(q)[Q_1,Q_1^{-1}]$, which
is nonzero for finitely many $(p_2,\dots,p_{\ell})$.

\begin{Cor}\label{shubu}
In Proposition \ref{kurikaeshi},
we further assume that we have
\begin{align}
\sum_{(j_2,\dots,j_{\ell})\in S_i}
\sigma^{(p_2,\dots,p_{\ell})}_{i,j_2,\dots,j_{\ell}}q^{j_2p_2+\dots+j_{\ell}p_{\ell}}=0
\label{addass}
\end{align}
in $\mathbb{Q}(q)[Q_1,Q_1^{-1}]$ for each $i\in\mathbb{Z}$ and $(p_2,\dots,p_{\ell})\in\mathbb{Z}^{\ell-1}\setminus\{(0,\dots,0)\}$, where 
\begin{align*}
  S_i = \{(j_2,\dots,j_{\ell})\in\mathbb{Z}^{\ell-1}\mid
  (i,j_2,\dots,j_{\ell})\in S\}.
\end{align*}
We have
a $q$-holonomic recurrence 
\begin{align}
0=\sum_{i\in\mathbb{Z}}h_{n-i}\sum_{(j_2,\dots,j_{\ell})\in S_i}
\sigma^{(0,\dots,0)}_{i,j_2,\dots,j_{\ell}}.
\label{reclocal}
\end{align}
\end{Cor}

\begin{proof}
Recall $\NONC$ in \S\ref{certe} and put $r=\ell$.
By Proposition \ref{kurikaeshi}, we have
\begin{align*}
\sum_{(i,j_2,\dots,j_{\ell})\in S}\sigma_{i,j_2,\dots,j_{\ell}}(q^n,q^{k_2},\dots,q^{k_{\ell}})
\ENU^{i}\KEI{2}^{j_2}\cdots\KEI{\ell}^{j_{\ell}}
=0
\end{align*}
in $\NONC$. Note that the left hand side is equal to 
\begin{align*}
\sum_{(i,j_2,\dots,j_{\ell})\in S}
  \sum_{(p_2,\dots,p_{\ell})\in \mathbb{Z}^{\ell-1}}
  \sigma^{(p_2,\dots,p_{\ell})}_{i,j_2,\dots,j_{\ell}}(q^n)q^{p_2k_2+\dots+p_{\ell}k_{\ell}}\ENU^{i}\KEI{2}^{j_2}\cdots\KEI{\ell}^{j_{\ell}},
\end{align*}
which is, by \eqref{noncrel}, equal to
\begin{align}
  \sum_{\substack{(i,j_2,\dots,j_{\ell})\in S \\ (p_2,\dots,p_{\ell})\in \mathbb{Z}^{\ell-1}}} \sigma^{(p_2,\dots,p_{\ell})}_{i,j_2,\dots,j_{\ell}}(q^n)q^{p_2j_2+\dots+p_{\ell}j_{\ell}}\ENU^{i}\KEI{2}^{j_2}\cdots\KEI{\ell}^{j_{\ell}}q^{p_2k_2+\dots+p_{\ell}k_{\ell}}.
  \label{temp9}
\end{align}
Thanks to the obvious identity $1-x^n=(1-x)(1+x+\dots+x^{n-1})$ for $n\geq 1$, there exist 
$T_2,\dots,T_{\ell}\in\NONC$ such that \eqref{temp9} is equal to 
\begin{align*}
\Delta_2T_2+\dots+\Delta_{\ell}T_{\ell}+
\sum_{(p_2,\dots,p_{\ell})\in \mathbb{Z}^{\ell-1}}
\sum_{(i,j_2,\dots,j_{\ell})\in S}
\sigma^{(p_2,\dots,p_{\ell})}_{i,j_2,\dots,j_{\ell}}(q^n)q^{p_2j_2+\dots+p_{\ell}j_{\ell}}\ENU^{i}q^{p_2k_2+\dots+p_{\ell}k_{\ell}},
\end{align*}
where $\Delta_j=1-\KEI{j}$ (see ~\cite[\S4]{Rie}).
Thus, if \eqref{addass} holds,
we get the claim by Proposition \ref{certprop} thanks to the obvious identity
\begin{align*}
  h_n=
  \sum_{(i_2,\dots,i_{\ell})\in\mathbb{Z}^{\ell-1}} F(n,i_2,\dots,i_{\ell}).
\end{align*}
Note that, for $n\in\mathbb{Z}$,
$F(n,i_2,\dots,i_{\ell})$ is nonzero for finitely many $(i_2,\dots,i_{\ell})\in\mathbb{Z}^{\ell-1}$ 
by the positivities of $c_2,\dots,c_{\ell}$.
\end{proof}

\begin{Ex}\label{reicertexp}
In Example \ref{exsister}, we have
\begin{align*}
T^{(S)}_{0,0}q^2Q_1^{-4}(1-Q_1^2)
-T^{(S)}_{1,0}qQ_1^{-2}
-T^{(S)}_{1,1}=0.
\end{align*}
Note that the additional assumption \eqref{addass}
in Corollary \ref{shubu}
is vacuously true because each $\sigma_{(i,j_2)}\in\mathbb{Q}[Q_1,Q_1^{-1},Q_2,Q_2^{-1}]$ is $Q_2$-free. By the division algorithm, we have
\begin{align*}
  q^2Q_1^{-4}(1-Q_1^2)-qQ_1^{-2}\ENU-\ENU\KEI{2}
  =
  q^2Q_1^{-4}(1-Q_1^2)-(1+qQ_1^{-2})\ENU+(1-\KEI{2})\ENU.
\end{align*}
By Proposition \ref{certprop} (or by Corollary \ref{shubu} without the above rewritten), we have 
\begin{align*}
  q^{2-4M}(1-q^{2M})h_M=(1+q^{-2M+1})h_{M-1}
\end{align*}
for $M\in\mathbb{Z}$.
This is equivalent to 
$(1-q^{2M})h_M=(q^{2M-1}+q^{4M-2})h_{M-1}$,
i.e., 
\begin{align*}
H(x,q)=(1+xq)H(xq^2,q)+xq^2H(xq^4,q).
\end{align*}
\end{Ex}

To summarize, the algorithm to obtain a $q$-difference equation
of \eqref{agh} is as follows.
For examples of certificate recurrences, see ~\cite{Ts1} and the proof of Proposition \ref{gs111}.
Note that ~\cite[Theorem 2.2]{Rie} guarantees that
(Step 3) below gives a nontrivial relation for large enough $S$
and $D=\{(0,\dots,0)\}$. For an effective bound on $S$,
see ~\cite[Proof of Theorem 2.2]{Rie}.

\begin{enumerate}
\item[(Step 1)] As an input, give finite subsets
  $S\subseteq\mathbb{Z}_{\geq0}^{\ell}$ and
  $D\subseteq\mathbb{Z}^{\ell-1}$ such that $(0,\dots,0)\in S$ and
  $(0,\dots,0)\in D$ heuristically.
\item[(Step 2)] Solve the equation \eqref{algoeqmain} and \eqref{addass} under the ansatz
  \begin{align*}
    \sigma_{i,j_2,\dots,j_{\ell}}=\sum_{(p_2,\dots,p_{\ell})\in D}
    \sigma^{(p_2,\dots,p_{\ell})}_{i,j_2,\dots,j_{\ell}}Q_2^{p_2}\cdots Q_{\ell}^{p_{\ell}}
\end{align*}
    for $(i,j_2,\dots,j_{\ell})\in S$
  (i.e., solve the simultaneous linear equation on
  the unknowns $(\sigma_{i,j_2,\dots,j_{\ell}}^{(p_2,\dots,p_{\ell})})_{(i,j_2,\dots,j_{\ell})\in S, (p_2,\dots,p_{\ell})\in D}$ over
  $\mathbb{Q}(q,Q_1)$).
\item[(Step 3)] If the relation \eqref{reclocal} is nontrivial, transform it to a nontrivial $q$-difference equation of $H(x,q)$. Otherwise, try (Step 1) for different $S$ and $D$.
\end{enumerate}


\subsection{Andrews-Gordon type series for $G_{(1,1,1)}(x,q)$ and $G_{(3,0,0)}(x,q)$}\label{gs111}

\begin{Prop}\label{g111}
We have
\begin{align*}
G_{(1,1,1)}(x,q)=\sum_{a,b,c,d\geq 0}\frac{q^{a^2+b^2+3c^2+3d^2+2ab+3ac+3ad+3bc+3bd+6cd}}{(q;q)_a(q;q)_b(q^3;q^3)_c(q^3;q^3)_d}x^{a+b+c+2d}.
\end{align*}
\end{Prop}

\begin{proof}
First, we show that the Andrews-Gordon type series in Proposition \ref{g111} 
satisfies the same $q$-difference equation for $G_{(1,1,1)}(x,q)$ in Proposition \ref{PropqdiffG}.
Let
\begin{align*}
F(n,b,c,d)=
\frac{q^{(n-b-c-2d)^2+b^2+3c^2+3d^2+2(n-b-c-2d)b+3(n-b-2c-2d)(c+d)+3bc+3bd+6cd}}{(q;q)_{n-b-c-2d}(q;q)_b(q^3;q^3)_c(q^3;q^3)_d}
\end{align*}
and put
\begin{align*}
p_0 &= -2q^{6n+8}-q^{3n+12}+2q^{3n+8}+q^{12}, \quad
r_0 = 0,\quad
s_0 = (2q^{3n+8}+q^{12})(q^b-1),\\
q_0 &= (4q^{4n+9}+2q^{4n+8}+q^{n+12}+2q^{n+10})(q^{2n}-q^{b+2c+d})+(2q^{3n+8}+q^{12})(q^{n+2c+d}-q^b),\\
p_1 &= -(2q^{6n}+3q^{3n+4}+4q^{3n+3}+4q^{3n+2}+2q^{3n+1}+2q^{6}+2q^{5}+2q^{4})q^{3n+5},\\
q_1 &= (4q^{6n+6}+2q^{6n+5}+q^{3n+9}+2q^{3n+7})(q^{3n}C+q^{2+b}+(B-Cq^b)q^{n+2+2c+d})\\
&\quad+ (4q^{7n+8}+2q^{7n+7}-q^{4n+11}+2q^{4n+9}-q^{n+12})q^{b+2c+d}\\
&\quad+ q^{n+7}(2q^{3n}+q^{4})((1+C)q^{3n}+Dq^{1+b})q^{2c+d}+ (2q^{3n+3}+2q^{3n+1}+2q^{3n}+q^{5}+2q^{4})q^{3n+6},\\
r_1 &= (2q^{9n+5}+q^{6n+9}),\quad
r_2 = (2q^{9n+3}+q^{6n+4})-(2q^{7n+4}+q^{4n+8})q^{b+2c+d},\\
s_1 &= (4q^{6n+8}+2q^{6n+7}+q^{3n+11}+2q^{3n+9})(1-q^b) + (2q^{3n}+q^4)q^{8+n+b+2c+d}, \\
p_2 &= 2q^{9n+4}+2q^{9n+3}-3q^{6n+8}+2q^{6n+5}+q^{6n+4}-q^{3n+9}, \\
q_2 &= ((4q^{10n+4}+2q^{10n+3}+q^{7n+7}+2q^{7n+5})(B+q^b)+2q^{10n+3}+q^{7n+7}-q^{4n+8+b}-2q^{7n+7+b})Cq^{2c+d}\\
&\quad +(-2q^{9n+4}+q^{6n+8}-2q^{6n+5})+2(q^{7n+7}+q^{7n+4}+q^{4n+8})q^{b+2c+d} \\
&\quad +(2q^{6n+7}+q^{3n+8})(q^{3n-4}C+Bq^{n+2c+d}+q^{n+2c+d}+q^{1+b}),\\
s_2 &= (2q^{6n+7}+q^{3n+8})(q(1-q^b)-2q^{n+b+2c+d}),\quad
r_3 = (4q^{3n}+2q)q^{7n+b+2c+d},\quad
s_3 = 0,\\
p_3 &= -(2q^{9n+2}+q^{6n+3}), \quad q_3 = (2q^{9n}+q^{6n+1})(q^2+(1+q^b+B)Cq^{n+2c+d}).
\end{align*}

Let $\TTE=\varepsilon_0+\varepsilon_1N+\dots+\varepsilon_3N^3$ for $\varepsilon\in\{p,q,r,s\}$.
One can check 
\begin{align}
(\TTP+(1-B)\TTQ+(1-C)\TTR+(1-D)\TTS)F(n,b,c,d)=0,
\label{onecancheck}
\end{align}
where $NF(n,b,c,d)=F(n-1,b,c,d)$, $BF(n,b,c,d)=F(n,b-1,c,d)$, $CF(n,b,c,d)=F(n,b,c-1,d)$ and $DF(n,b,c,d)=F(n,b,c,d-1)$.
By Proposition \ref{certprop}, we have
\begin{align*}
p_0f_n+p_1f_{n-1}+p_2f_{n-2}+p_3f_{n-3}=0.
\end{align*}

On the other hand, the $q$-difference equation for $G_{(1,1,1)}(x,q)$ in Proposition \ref{PropqdiffG}
is equivalent to the claim that
\begin{align*}
p'_0g_n+p'_1g_{n-1}+p'_2g_{n-2}+p'_3g_{n-3}+p'_4g_{n-4}=0.
\end{align*}
holds for all $n\in\mathbb{Z}$. Here, 
\begin{align*}
p'_0 &= -1+q^{3n},\quad p'_1 = -q^{4}+2q^{3n-2}+2q^{3n-1}+2q^{3n}+q^{3n+1}+q^{6n-3},\\
p'_2 &= q^{3n-3}+2q^{3n-2}+2q^{3n-1}+2q^{3n}+q^{6n-8}-q^{6n-5}-q^{6n-4},\\
p'_3 &= q^{3n-2}-q^{6n-10}-q^{6n-9}+q^{6n-6},\quad p'_4 = q^{6n-11},
\end{align*}
and $g_{n}(q)$ is defined by $G_{(1,1,1)}(x,q)=\sum_{n\in\mathbb{Z}} g_{n}(q)x^n$.
One can check that
\begin{align*}
\frac{1}{-2q^{3n+8}-q^{12}}(p_0,p_1,p_2,p_3,0)
+
\frac{1}{-2q^{3n+4}-q^{8}}(0,\ENU p_0,\ENU p_1,\ENU p_2,\ENU p_3)
\end{align*}
is equal to $(p'_0,p'_1,p'_2,p'_3,p'_4)$, which implies that
the both $q$-difference equations are the same.
Finally, the claim follows from $g_{n}=f_n=0$ for $n<0$, $g_0=f_0=1$ and $p_0\ne 0$, $p'_0\ne 0$ for $n\geq 1$.
\end{proof}


In the rest of this section, 
we omit the display of explicit certificate recurrences
because of the following four reasons: they can be automatically get, 
they are large,
they seem to give us no insight, and
we have been provided them in the first version of this paper on arXiv.
Instead, we specify our input $(S,D)$ to the algorithm.
This high level specification guarantees that
certificate recurrences can be reproduced if desired.
For example, the certificate recurrence in Example \ref{reicertexp}
is obtained by putting $(S,D)=(S_{2,2},D_{1})$, where
$Y_{a_1,...,a_b}=\{(c_1,\dots,c_b)\in\mathbb{Z}^b\mid 0\leq c_j< a_j\}$ for $Y\in\{S,D\}$, and \eqref{onecancheck} is obtained 
by putting $(S,D)=(S_{4,3,2,2},D_{2,3,2})$.

\begin{Prop}\label{g300}
We have
\begin{align*}
G_{(3,0,0)}(x,q)=\sum_{a,b,c,d\geq 0}\frac{q^{a(a+1)+b(b+2)+3c(c+1)+3d(d+1)+2ab+3ac+3ad+3bc+3bd+6cd}}{(q;q)_a(q;q)_b(q^3;q^3)_c(q^3;q^3)_d}x^{a+b+c+2d}.
\end{align*}
\end{Prop}

\begin{proof}
The argument is the same as in \S\ref{gs111}
by $(S,D)=(S_{4,3,2,2},D_{2,3,2})$ after swapping $a$ and $b$.
As declared in \S\ref{certe}, we omit a detail of the automatic proof.
\end{proof}

\subsection{A proof of Theorem \ref{RRidentAG} for $\BIR$}\label{proof}
As in \S\ref{gs111}, the key is to show that the Andrews-Gordon type series in  Theorem \ref{RRidentAG} for $\BIR$ 
satisfies the same $q$-difference equation for $f_{\BIR}(x,q)$ in Proposition \ref{PropqdiffR}.
As declared in \S\ref{certe}, we omit a detail of the automatic derivation of
the former only saying that it is enough to put $(S,D)=(S_{7,3,2,2},D_{2,2,2})$
in the algorithm.

\subsection{A proof of Theorem \ref{RRidentAG} for $\BIRP$}\label{proofp}
The argument is the same as in \S\ref{proof}
by $(S,D)=(S_{7,2,2,2},D_{2,3,3})$ after swapping $a$ and $b$.
As declared in \S\ref{certe}, we omit a detail of the automatic proof.

\subsection{A proof of Theorem \ref{RRbiiden}}\label{finalsec}
By Proposition \ref{g111}, Proposition \ref{g300} and Theorem \ref{RRidentAG}, we have
$f_{\BIR}(q)=G_{(3,0,0)}(1,q)$ and $f_{\BIRP}(q)=G_{(1,1,1)}(1,q)$.
Thanks to \eqref{charcalc} and Theorem \ref{borp}, we get the results. 

\hspace{0mm}

\noindent{\bf Acknowledgments.} 
The author was supported by
the Research Institute for Mathematical
Sciences, an International Joint Usage/Research Center located in Kyoto
University,
JSPS Kakenhi Grants 20K03506, 23K03051, Inamori Foundation, JST CREST Grant Number JPMJCR2113, Japan and Leading Initiative for Excellent Young Researchers, MEXT, Japan.






\end{document}